\documentclass[11pt]{amsart}
\usepackage[utf8]{inputenc}
\usepackage[english]{babel}
\usepackage{cite}
\usepackage{breqn}
\usepackage{amsthm,amsmath,amssymb,hyperref,amsfonts,tikz,adjustbox}
\usepackage{mathrsfs}
\usetikzlibrary{cd}
\numberwithin{equation}{section}
\usepackage[shortlabels]{enumitem}
\usepackage[a4paper, total={6in, 8.5in}]{geometry}
\usepackage{imakeidx}
\usepackage{breqn}
\usepackage{ulem}
\usepackage{comment}

\usepackage{xcolor} 

\usepackage{tikz}
\usetikzlibrary{cd}

\usetikzlibrary{calc} 
\usetikzlibrary{decorations.pathmorphing} 

\tikzset{curve/.style={settings={#1},to path={(\tikztostart)
    .. controls ($(\tikztostart)!\pv{pos}!(\tikztotarget)!\pv{height}!270:(\tikztotarget)$)
    and ($(\tikztostart)!1-\pv{pos}!(\tikztotarget)!\pv{height}!270:(\tikztotarget)$)
    .. (\tikztotarget)\tikztonodes}},
    settings/.code={\tikzset{quiver/.cd,#1}
        \def\pv##1{\pgfkeysvalueof{/tikz/quiver/##1}}},
    quiver/.cd,pos/.initial=0.35,height/.initial=0}

\tikzset{tail reversed/.code={\pgfsetarrowsstart{tikzcd to}}}
\tikzset{2tail/.code={\pgfsetarrowsstart{Implies[reversed]}}}
\tikzset{2tail reversed/.code={\pgfsetarrowsstart{Implies}}}
\tikzset{no body/.style={/tikz/dash pattern=on 0 off 1mm}}

\newtheorem{theorem}{Theorem}[section] 
\newtheorem{corollary}[theorem]{Corollary}
\newtheorem{lemma}[theorem]{Lemma} 
\newtheorem{proposition}[theorem]{Proposition}

\theoremstyle{definition}
\newtheorem{definition}[theorem]{Definition}
\newtheorem{example}[theorem]{Example}

\theoremstyle{remark}
\newtheorem{remark}[theorem]{Remark}

\newcommand*{\im}{\operatorname{im}} 
\newcommand*{\en}{\operatorname{End}}

\newcommand{\der}{\operatorname{Der}}

\newcommand{\F}{\mathcal{F}}
\newcommand{\Hom}{\operatorname{Hom}}

\title[Globalization of Partial group actions on non-associative algebras]{Globalization of Partial Group Actions on Not Necessarily Associative Algebras and Covariant Representations}

\keywords{Partial group actions, partial group representations, non-associative algebras, varieties of algebras}
\subjclass[2020]{Primary 16W22, 17A01 , Secondary 08B99}

\author{Mikhailo Dokuchaev}
\address[M. Dokuchaev]{Departamento de Matem{\'a}tica, Universidade de S\~ao Paulo,
Rua do Mat\~ao, 1010, 05508-090 S\~ao Paulo, Brazil}
\email{dokucha@ime.usp.br}  
\author{Emmanuel Jerez}
\address[E. Jerez]{Guangdong Technion Israel-Institute of Technology, Shantou, Guangdong
Province, China}
\email{ars.ejerez@icloud.com}  
\author{José L. Vilca-Rodríguez}
\address[J. L. Vilca-Rodríguez]{Departamento de Matem{\'a}tica, Universidade de S\~ao Paulo,
Rua do Mat\~ao, 1010, 05508-090 S\~ao Paulo, Brazil}
\email{jvilca@ime.usp.br}

\begin{document}

\begin{abstract}
     We extend the concept of a partial group action to non-associative algebras in a variety \(\mathcal{V}(I)\), solve the globalization problem  within \(\mathcal{V}(I)\) and examine its universal property. It is achieved
     using what we call the ``$\Lambda$-construction'', which we also apply to deal with covariant representations in the associative and Lie algebra settings, considering related categories and constructing an adjoint pair of functors between them.  We also show that the $\Lambda$-construction behaves well with semidirect products of Lie algebras.
\end{abstract}

\maketitle

\tableofcontents

\section*{Introduction}

Partial group actions are a flexible and powerful generalization of classical (global) group actions, they allow a group to act only on ``pieces'' of a mathematical object while retaining much of the formalism and applications of global symmetry. 
They can be produced  by restricting  global actions to not necessarily invariant subobjects. For example,  the projective linear group 
${\rm PGL}(n, {\mathbb C})$ acts globally on the Riemann sphere but only partially on the complex plane ${\mathbb C}$ (see \cite{Article_Kellendonk-Lawson_2004_PAOG}).   This naturally suggests  the question if  a given   partial action can be seen as a restriction of a global one. Under the name of ``maximum local transformation groups'' smooth partial actions of Lie groups on not necessarily Hausdorff manifolds were studied from the globalization point of view   by R. Palais  
\cite[Theorem X, p. 72]{Article_Palais_1957_AGFOTTOTG}. Later, the globalization problem  was independently studied by F. Abadie \cite{Article_Abadie_2003_EAADFPA}  for partial  actions  of  groups on 
topological spaces and on $C^*$-algebras  and by J. Kellendonk and M. V. Lawson   for partial group  actions   on topological spaces and  semilattices \cite{Article_Kellendonk-Lawson_2004_PAOG}. Since then constant attention is payed to  the globalization problem of partial actions and partial coactions:
\cite{Article_Steinberg_2003_PAOGOCC, Article_Megrelishvili-Shroder_2004_Globalization,
Gilbert_2005_Actions,
Ferrero_2006_ParAcSemiprime,
Dokuchaev-DelRio-Simon_2007_Glob,Hollings_2007_PartialActions,Cortes-Ferrero_2009_GlobSemiprime,Article_Gould-Hollings_2009_PartialActions,Article_Alves-Batista_2010_EAFPHA,Article_Dokuchaev-Exel-Simon_2010_GOTPA,Article_Alves-Batista_2011_GTFPHASOTA,Bagio-Paques_2012_Glob,Article_Alvares-Alves-Batista_2013_PHMC,Bemm-Ferrero_2013_Glob,Article_Kudryavtseva_2015_PMAAACORS,Article_Castro-Paques-Cuadros-Sant'Ana_2015_PAOWHASPGAMT,Article_Alves-Batista-Dokuchaev-Paques_2016_GOTPHA,Bagio-Pinedo_2016_Glob,Cortes-Ferrero-Marcos_2016_Categories,Khrypchenko-Novikov_2018_Reflectors,Kraken-Quast-Timmermann_2018_PartialQuantum,Ferraro_2018_ConstructionGlobalization,Nystedt_2018_Partial_Category,Bagio-Paques-Pinedo_2020_RestrExt,Fonseca-Fontes-Martini_2020_Multiplier,Marin-Pinedo_2020_Partial_Groupoid_R-categories,Fontes-Martini-Fonseca_2022_WeakHopf,Article_Saracco-Vercruysse_2022_GFGPC,Article_Saracco-Vercruysse_2022_OTGOGPITCOTSAA,Article_Kudryavtseva-Laan_2023_Globalization,Alves-Ferraza_2024_HopfMoritaAndGlobNonunital,Batista_Castro_Khrypchenko_2024_Monoidal_arXiv,Cortes-Wagner_2024_GlobNonAssoc,Article_Jerez_2024_OTHOPGA,Khrypchenko-Klock_2024_ParGlob,Lautenschlaeger-Tamusiunas_2024_Glob_arXiv,Article_Saracco-Vercruysse_2024_GPCOFCIACAG,Marin-Pinedo-Rodriguez_2025_Partial_Groupoid,Demeneghi-Tasca_2025_Globalization,Article_Demeneghi-Marin-Tasca-Velasco_2025_P-Theorem,Batista-Hautekiet-Vercruysse_2025_Glob_arXiv}.
The initial construction of the globalization  was given for    partial group actions on smooth manifolds  \cite{Article_Palais_1957_AGFOTTOTG},   topological spaces  \cite{Article_Abadie_2003_EAADFPA}, \cite{Article_Kellendonk-Lawson_2004_PAOG}, semilattices \cite{Article_Kellendonk-Lawson_2004_PAOG} and $2$-complexes \cite{Article_Steinberg_2003_PAOGOCC}. The idea  is essentially the same and, as it was pointed out by J. Kellendonk and M. Lawson \cite{Article_Kellendonk-Lawson_2004_PAOG}, it  appeared in a wide variety of areas in the proofs of remarkable results, which can be reformulated in terms of partial group actions. It is a ``partial'' version of the tensor product of $G$-sets of the form $G\otimes _G X.$     The idea is  natural and efficient,  and it was applied with eventual adaptations to produce globalizations in a series of papers.   

For unital associative algebras a construction was given in \cite{Article_Dokuchaev-Exel_2005_AOCPBPAEAAPR}, which is an appropriate subalgebra of the direct power $A^G,$  where 
the group $G$ acts permuting the direct factors. As in the case of the $G\otimes _G X$-tool, this idea was further used in a number of articles with  appropriate adjustments and interesting developments. 

   In \cite{Article_Jerez_2024_OTHOPGA} a construction was given for the globalization of partial group actions on modules,
  which is, loosly speeking, a linearization of the ``partial''  $G\otimes _G X,$ and produced by  a combination of tensor products of modules and that of $G$-sets (see Section~\ref{Generalized partial group actions}). We refer to it as the $\Lambda $-construction. 
  
  In this paper we further  explore the $\Lambda $-construction not only for the globalization of partial actions, but also to deal with covariant representations of partial representations and their dilations.  
  More specifically,  we extend the theory of partial group actions to the setting of non-associative algebras belonging to an arbitrary variety \(\mathcal{V}(I)\) and develop the corresponding notions of partial representations, globalizations, and covariant representations. Our approach follows two complementary threads.
First, given a partial representation \(\pi\) of a group \(G\) by endomorphisms of an algebra \(A\in \mathcal{V}(I)\) (with images that are two-sided ideals), we construct a canonical \(K\)-module \(\Lambda(A)\) that carries a natural algebra structure and a global \(G\)-action induced from \(\pi\).
This construction makes \(\Lambda(A)\) into a non-associative algebra in \(\mathcal{V}(I)\),
and moreover, it acts as a globalization of the partial action induced by \(\pi\) on \(A\).
This result not only generalizes the classical globalization theorem for unital partial group action on unital associative algebras  proved in \cite{Article_Dokuchaev-Exel_2005_AOCPBPAEAAPR}, but also provides a systematic method to globalize partial actions in a wide variety of algebraic contexts.
Second, we consider generalized partial actions of \(G\) on an algebra \(A\in \mathcal{V}(I)\) (where the images of the partial action are subalgebras, not necessarily ideals) and show that the \(\Lambda\)-construction yields a universal globalization \(\Omega(A)\) for such actions.
An importan example that illustrates the deep connection between both threads is the case of partial actions on Lie algebras.
In Example~\ref{e: partial action on U(L) 2} we show the interplay between the two globalization constructions in the case of a partial action on a Lie algebra \(L\) and its induced partial representation on the universal enveloping algebra \(U(L)\).

Next, we study partial covariant representations  in the  associative and Lie algebra settings and the functorial behavior of the globalization process.
For a partial representation  \(\lambda\colon G\to\mathrm{End}_0 (A)\), where $A$ is an associative or Lie algebra and $\mathrm{End}_0 (A)$ denotes  the algebra endomorphisms of $A$,   we define a category of partial covariant representations and show how the \(\Lambda\)-construction lifts these to honest (global) covariant representations of \(\Lambda(\lambda)\).
We then isolate a natural subcategory of ``Dilations'' and prove an adjunction between the lifted functor \(\Lambda\) and a functor from \textit{Dilations} to covariant representations, thereby clarifying the universal character of the construction at the level of representations.

Finally, we verify that the globalization behaves well with respect to semidirect products: when  the Lie algebra \(L\) acts by derivations on the Lie algebra \(M\) and the pair of partial representations \(\lambda\) on \(L\) and \(\pi\) on \(M\) of a group $G$ form a partial covariant representation, the \(\Lambda\)-construction commutes  with  semidirect products; concretely,
\[
\Lambda(L\!\ltimes\!M)\ \cong\ \Lambda(L)\!\ltimes\!\Lambda(M),
\]
and the induced group actions are naturally isomorphic.

\section{Preliminaries}

Let \( K \) be  an associative commutative unital ring. Throughout this work, \( K \) is fixed as the ground ring. Unless stated otherwise, all modules considered are \( K \)-modules,  and we denote the tensor  product  over $K$ of
two \(K\)-modules \(X\) and \(Y\) simply as \(X\otimes Y\).  The capital letter \( G \) will denote a group, and the identity element of an algebraic structure will be denoted by \( 1 \), provided no ambiguity arises. As it is usual, by a non-associative algebra we mean a not necessarily associative algebra.

This work focuses on partial actions and partial representations of groups. For simplicity, the terms \emph{partial action} and \emph{partial representation} will refer specifically to partial group actions and partial group representations, respectively. 

We begin by recalling some foundational concepts related to varieties of algebras, partial actions, and partial representations.

\subsection{Varieties of algebras}
Let $X$ be a nonempty countable set. A  non-associative word of length $n$ on $X$  is a juxtaposition of $n$-elements of $X$  with a specific way of placing parentheses. For instance, if $X=\{a\}$ then $a,aa,a(aa),(aa)a$ are non-associative words of length $\leq 3$. We can multiply two  non-associative words $v, w$ by juxtaposition but we must include parentheses around each factor. For example, the product of $a(aa)$ and $(aa)a$ is the non-associative word $(a(aa))((aa)a)$.  Let us denote by $F(X)$  the vector space with a basis consisting of all non-associative words in $X$. That is, $F(X)$ consists of all non-associative polynomials in $X$ with coefficients in $K$. By extending the product of non-associative words to $F(X)$ via distributivity, we define a multiplication in 
$F(X)$ that makes it a non-associative algebra (without a unit element). The algebra $F(X)$ is called the {\bf free non-associtive algebra} generated by $X$. There is a canonical inclusion $j: X\to F(X)$, and $F(X)$ verifies the following universal property: Given a non-associtive algebra $A$ and a  map $\phi:X\to A$, there exists a unique algebra homomorphism $\widehat{\phi}: F(X)\to A$ such that $\widehat{\phi}\circ j=\phi$.

Let \( A \) be a non-associative algebra. A non-associative polynomial \( f(x_1, \ldots, x_n) \in F(X) \) is called an \textbf{identity} of \( A \) if \( f(a_1, \ldots, a_n) = 0 \) for all \( a_1, \ldots, a_n \in A \). Let \( I \) be a subset of \( F(X) \). The class \( \mathcal{V}(I) \) of all  non-associative algebras satisfying all identities in \( I \) is called the \textbf{variety defined by \( I \)}, and it is clearly a category. One can construct a free algebra in \( \mathcal{V}(I) \) generated by a nonempty set \( Y \) as follows. Let \( T(Y, I) \) be the set of images of \( I \) in \( F(Y) \) under all homomorphisms from \( F(X) \) into \( F(Y) \). Consider the ideal \( J(Y, I) \) in \( F(Y) \) generated by \( T(Y, I) \). Then the algebra \( F_I(Y) = F(Y) / J(Y, I) \) belongs to the variety \( \mathcal{V}(I) \). If \( \nu \) denotes the restriction to \( Y \) of the canonical homomorphism \( F(Y) \to F_I(Y) \), given by \( b \mapsto b + J(Y, I) \), then one can verify that \( F_I(Y) \) satisfies the following universal property: Given a non-associative algebra \( A \in \mathcal{V}(I) \) and a map \( \phi: Y \to A \), there exists a unique algebra homomorphism \( \widehat{\phi}: F_I(Y) \to A \) such that \( \widehat{\phi} \circ \nu = \phi \). The algebra \( F_I(Y) \) is called the \textbf{free algebra generated by \( Y \)} in \( \mathcal{V}(I) \). If the variety \( \mathcal{V}(I) \) is non-trivial, i.e. it contains a nonzero algebra, then the map \( \nu: Y \to F_I(Y) \) is injective, and we can view \( Y \) as a subset of \( F_I(Y) \) (see~\cite[p.~26]{Book_Jacobson_1968_SRJA}). 

 Let $Y$ be a set and denote by $K^{(Y)}$ the  free \( K \)-module with basis $Y.$ Then $K^{(Y)}$ can be seen as a $K$-submodule of the algebra  $ F(Y)$ and the map \( \nu: Y \to F_I(Y) \)  extends to a $K$-linear map    \( i: K^{(Y)} \to F_I(Y) \), which is the restriction of  \( F(Y) \to F_I(Y) \) to $K^{(Y)}.$
Then the above universal property of $ F_I(Y)$ implies that  for any  algebra \( A \in \mathcal{V}(I) \) and any linear map \( \phi: K^{(Y)} \to A \), there exists a unique algebra homomorphism \( \widehat{\phi}: F_I(Y) \to A \) such that \( \widehat{\phi} \circ i = \phi \). As above, the $K$-linear map    \( i \) is injective if  \( \mathcal{V}(I) \)  is non-trivial. For  let $A$ be  a non-zero algebra in  \( \mathcal{V}(I) \) and write 
$B = \oplus _{y\in Y} A_y,$ where $A_y$ is a copy  of  $A.$ Denote by $a_y$ a  non-zero element of $ A_y.$ Then the map $Y\to B$ such that $y\mapsto a_y$  extends to an injective $K$-linear map $\phi : K^{(Y)} \to B,$ and by the above property there exists an  algebra homomorphism \( \widehat{\phi}: F_I(Y) \to B \) such that \( \widehat{\phi} \circ i = \phi \). Since $\phi $ is injective, so is  $i.$

Let $ M $ be an arbitrary \( K \)-module and $Y$ a generating set of $M.$ Then the embedding $Y\to M$ extends to an epimorphism $K^{(Y)} \to M.$ Let $T_M$ be the kernel of this epimorphism and denote by $J_M$ the ideal of  $F_I(Y)$ generated by 
 $i (T_M).$  Then the algebra 
 $$F_I(M): = F_I(Y)/ J_M$$
 belongs to the variety $\mathcal{V}(I),$ and the composition of \( i: K^{(Y)} \to F_I(Y) \) with the quotient map $F_I(Y) \to F_I(M)$ vanishes on 
 $T_M,$ so that it defines a linear map \( j: M \to F_I(M) \). It is readily seen that the following universal property is satisfied:
  Given an algebra \( A \in \mathcal{V}(I) \) and a linear map \( \phi: M \to A \), there exists a unique algebra homomorphism \( \widehat{\phi}: F_I(M) \to A \) such that  \( \widehat{\phi} \circ j = \phi \).  It follows that  up to isomorphism  $F_I(M)$ is  the unique  algebra in $\mathcal{V}(I)$ which satisfies this property   and, in particular, 
  $F_I(M)$ does not depend on the choice of the generating set $Y$ of $M.$
 We refer to \( F_I(M) \) as the \textbf{free algebra generated by \( M \)} in \( \mathcal{V}(I) \). 
 
 \begin{remark}\label{rem:injective_j} Suppose that  $\mathcal{V}(I)$ is a variety which contains the algebras with zero multiplication. Then for any $K$-module $M$ the above defined $K$-linear map $j: M \to F_I(M)$ is injective. Indeed, endowing $M$ with the zero multiplication and applying the universal property of $F_I(M)$ to the identity map $M\to M,$ we immediately obtain that $j$ has to be  injective.
 \end{remark}
 
 Note that Remark~\ref{rem:injective_j} is applicable to such classical varieties as associative algebras, Lie algebras, alternative algebras and Jordan algebras. More generally,  any non-trivial variety of algebras  defined by a system $I$ of 
homogeneous identities satisfies the condition of Remark~\ref{rem:injective_j}. This includes all non-trivial homogeneous varieties of algebras, in particular,  any non-trivial variety of algebras over an infinite  field (see \cite[p.8]{Book_Zhevlakov-Slinko-Shestakov-Shirshov_1982_RTNAR}).

\subsection{Generalized partial group actions}\label{Generalized partial group actions}

In what follows, \( \mathcal{V}(I) \) denotes a variety  of non-associative algebras defined by \( I \).

\begin{definition} \label{d: set partial group action}
    A (set-theoretic) \textbf{partial action} of a group \( G \) on a set \( X \) is a collection of data 
    \[
        \theta = \big( G, X, \{ X_{g} \}_{g \in G}, \{ \theta_{g} \}_{g \in G} \big),
    \]
    where \( \{ X_g \}_{g \in G} \) is a family of subsets \( X_g \subseteq X \) and \( \{ \theta_g \}_{g \in G} \) is a family of bijections \( \theta_g: X_{g^{-1}} \to X_g \) for each \( g \in G \), satisfying the following conditions:
    \begin{enumerate}[(i)]
        \item \( X_1 = X \) and \( \theta_1 = 1_X \) (the identity map on \( X \)),
        \item \( \theta_g(X_{g^{-1}} \cap X_{g^{-1}h}) \subseteq X_g \cap X_h \) for all \( g, h \in G \),
        \item \( \theta_g \theta_h(x) = \theta_{gh}(x) \) for all \( x \in X_{h^{-1}} \cap X_{h^{-1}g^{-1}} \).
    \end{enumerate}
    If every \( X_g \) coincides with \( X \), we say that \( \theta \) is a \textbf{global action} of \( G \) on \( X \).
\end{definition}

\begin{definition} \label{d: partial group action}
 Let \( A \) be a \( K \)-module, and let 
        \[
            \alpha = \big( G, A, \{ A_{g} \}_{g \in G}, \{ \alpha_{g} \}_{g \in G} \big)
        \]
        be set-theoretic partial action of \( G \) on \( A \). Then,
        \begin{enumerate}[(i)]
            \item \(\alpha\) is a \textbf{partial action of \(G\) on the \(K\)-module} \(A\) if each \(A_g\) is a submodule of \(A\), and each \(\alpha_g: A_{g^{-1}} \to A_g\) is a linear isomorphism.

            \item If \(A \in \mathcal{V}(I)\), each \(A_g\) is a bilateral ideal of \(A\), and each \(\alpha_g: A_{g^{-1}} \to A_g\) is an algebra isomorphism, then \(\alpha\) is called a \textbf{partial action of \(G\) on the algebra} \(A\).

            \item If \(A \in \mathcal{V}(I)\), each \(A_g\) is a subalgebra of \(A\), and each \(\alpha_g: A_{g^{-1}} \to A_g\) is an algebra isomorphism, then \(\alpha\) is called a \textbf{ generalized partial action of \(G\) on the algebra} \(A\).
        \end{enumerate}
    \end{definition}

For practical purposes, we will denote the (generalized) partial action \[ \alpha = \big( G, A, \{ A_{g} \}_{g \in G}, \{ \alpha_{g} \}_{g \in G} \big) \] by \( \alpha: G \curvearrowright A \).

A key example of a generalized partial action arises from the canonical restriction of a (global) group action on an algebra. Specifically, let \( B \in \mathcal{V}(I) \) and let \( \beta: G \to \mathrm{Aut}(B) \) be an action of \( G \) on \( B \). If \( A \) is a subalgebra of \( B \), then setting \( A_g = B \cap \beta_g(B) \) and defining \( \alpha_g \) as the restriction of \( \beta_g \) to \( A_{g^{-1}} \), we obtain a generalized partial action of \( G \) on \( A \). This generalized partial action is referred to as the \textbf{restriction} of \( \beta \) to \( A \). Furthermore, if \( A \) is an ideal of \( B \), the restriction of \( \beta \) to \( A \) is a partial action of \( G \) on \( A \).

\begin{definition}\label{equivariant map}
    Let \( A \) and \( B \) be \( K \)-modules or algebras in \( \mathcal{V}(I) \), and let \( \alpha: G \curvearrowright A \) and \( \beta: G \curvearrowright B \) be generalized partial actions. A \textbf{\( G \)-equivariant map} \( \phi: A \to B \) is a linear map (or homomorphism) that satisfies the following conditions:
    \begin{enumerate}[a)]
        \item \( \phi(A_g) \subseteq B_g \) for all \( g \in G \),
        \item \( \phi(\alpha_g(x)) = \beta_g(\phi(x)) \) for all \( g \in G \) and \( x \in A_{g^{-1}} \).
    \end{enumerate}
    Additionally, if \(\phi\) is bijective, and \(\phi^{-1}\) is also \(G\)-equivariant, then \(\phi\) is called an \textbf{isomorphism}.
\end{definition}

Clearly, the collection of all generalized partial actions of \( G \) forms a category, whose morphisms are the \( G \)-equivariant maps.


\begin{definition} \label{d: universal global partial action}
    Let \( M \) be a \( K \)-module, and let \( \alpha: G \curvearrowright M \) be a partial action on \( M \). A pair \( (\theta, \iota) \), where \(  \theta : G \curvearrowright W \) is a global action on a \( K \)-module, and \( \iota: M \to W \) is a \( G \)-equivariant map, is called: 
    \begin{enumerate}
        \item A \textbf{universal global action} of \( \alpha \) if, for any global action \( \beta: G \curvearrowright X \) on a \( K \)-module and any \( G \)-equivariant map \( \psi: M \to X \), there exists a unique \( G \)-equivariant map \( \widetilde{\psi}: W \to X \) such that the following diagram commutes:
            \[
                \begin{tikzcd}
                    M \arrow[r, "\psi"] \arrow[d, "\iota"'] & X \\
                    W \arrow[ur, dashed, "\widetilde{\psi}"']
                \end{tikzcd}
            \]

        \item A \textbf{globalization} of \( \alpha \) if \( \iota: M \to W \) determines an isomorphism between \( \alpha \) and the restriction of \( \theta \) to \( \iota(M) \), and \( W = \sum_{g \in G} \theta_g(\iota(M)) \). When a globalization of \( \alpha \) exists, we say that \( \alpha \) is \textbf{globalizable}.

    \end{enumerate}
\end{definition}

\begin{remark}
    It is important to note that a universal global action \( (\theta, \iota) \) for a partial action 
   \( \alpha: G \curvearrowright M \) is unique up to isomorphism.
\end{remark}

In \cite[Theorem 2.26]{Article_Jerez_2024_OTHOPGA}, one can find a construction of the universal global action for a partial action on a module.  In the next section, we exploit this result to obtain a universal globalization for a partial action on an algebra in \( \mathcal{V}(I) \). For convenience, we recall the construction presented in~\cite{Article_Jerez_2024_OTHOPGA}.

Let \( \alpha: G \curvearrowright M \) be a partial group action of \( G \) on a \( K \)-module \( M \). Consider the \( K \)-module \( \Lambda(M) \) defined as the quotient of \(  (KG \otimes M) /\mathcal{K}_{\alpha}\), where \(\mathcal{K}_{\alpha}\) is the \( K \)-submodule generated by
\[
    \{ g \otimes \alpha_h(m) - gh \otimes m : g, h \in G \text{ and } m \in M_{h^{-1}} \}.
\]
We denote the class of \( g \otimes m \) in \( \Lambda(M) \) by \( \lfloor g, m \rfloor \), and let \( \Lambda(\alpha): G \curvearrowright \Lambda(M) \) be the global group action defined as follows:
\begin{align} \label{eq: Theta}
    \begin{split}
        \Lambda(\alpha): G \times \Lambda(M) &\to \Lambda(M), \\
        (g, \lfloor h, m \rfloor) &\mapsto \lfloor gh, m \rfloor.
    \end{split}
\end{align}
Consider the linear map
\begin{align} \label{eq: iota equivariant map}
    \begin{split}
        \iota: M &\to \Lambda(M), \\
        m &\mapsto \lfloor 1, m \rfloor.
    \end{split}
\end{align}
By direct computation, one can verify that \( \iota \) is a \( G \)-equivariant map.

\begin{theorem}[\hspace{-3pt} Theorem 2.26\cite{Article_Jerez_2024_OTHOPGA}] \label{t: universal global action}
    Let \( \alpha: G \curvearrowright M \) be a partial group action on the \( K \)-module \( M \), and let \( \Lambda(\alpha): G \curvearrowright \Lambda(M) \) and \( \iota: M \to \Lambda(M) \) be as defined in \eqref{eq: Theta} and \eqref{eq: iota equivariant map}, respectively. Then, \( (\Lambda(\alpha), \iota) \) is the universal global action of \( \alpha \).
\end{theorem}

\begin{remark} Note that by the proof of \cite[Theorem 2.26]{Article_Jerez_2024_OTHOPGA} if   \( \beta: G \curvearrowright X \) is a global action on a \( K \)-module and  \( \psi: M \to X \) is a  
    \( G \)-equivariant map, then the map \( \widetilde{\psi}: \Lambda(M) \to X \) is determined by $ \widetilde{\psi} (\lfloor g, m \rfloor) = \beta_g ( \psi(m)), g\in G, m \in M.$
\end{remark}

The following lemma follows directly from the definition of \(\mathcal{K}_{\alpha}\)  and will be useful in some subsequent arguments. 

\begin{lemma}\label{l: computation tool for Lambda}
    Let \(f: KG \otimes M \to X\) be a linear map such that \(f(g \otimes \alpha_{h}(m)) = f(gh \otimes m)\) for all \(g, h \in G\) and \(m \in M_{h^{-1}}\). Then, there exists a unique linear map \(\widetilde{f}: \Lambda(M) \to X\) such that \(\widetilde{f}(\lfloor g, m \rfloor) = f(g \otimes m)\) for all \(g \in G\) and \(m \in M\).
\end{lemma}

\subsection{Partial group representations}

\begin{definition}\label{d: partial representation}
    Let \(G\) be a group and \(S\) be a monoid.
    A {\bf partial representation} of \(G\) into \(S\) is a map \(\pi: G \to S\) such that for all \(g, h \in G\), the following conditions are satisfied:
    \begin{enumerate}[(i)]
        \item \(\pi_1 = 1_S\); 
        \item \(\pi_{g^{-1}}\pi_{g}\pi_{h} = \pi_{g^{-1}}\pi_{gh}\);
        \item \(\pi_{g}\pi_{h}\pi_{h^{-1}} = \pi_{gh}\pi_{h^{-1}}\).
    \end{enumerate}
    For the sake of simplicity, we will denote \(\pi_g \pi_{g^{-1}}\) by \(\varepsilon_g\).
   We say that the partial representation $\pi $ is {\bf global} if 
     \( \pi_{g}\pi_{h} = \pi_{gh}\) for all $g,h \in G.$
\end{definition}

 Note that it immediately follows from Definition~\ref{d: partial representation} that  a partial representation
\(\pi: G \to S\) is global if and only if   
\(\varepsilon_g =1\) for all $g\in G.$

\begin{remark}\label{r: partial actions induced by partial representations}
    Partial actions and partial representations on modules are closely related.
    In fact, let $M$ be a $K$-module $M$, and  \(\pi: G \to \operatorname{End}_{K}(M)\) a partial representation,  where $\operatorname{End}_K(M)$ denote the set of all endomorphisms of the $K$-module $M$. Then we can define a partial action \(\alpha: G \curvearrowright M\) by setting \(M_g = \im \pi_g\) and \(\alpha_g = \pi_g|_{M_{g^{-1}}}\)  (see \cite[Example 2.3]{Article_Jerez_2024_OTHOPGA}).
    We refer to \(\alpha\) as the partial action induced by \(\pi\).
    Moreover, morphisms between partial representations naturally induce morphisms between the corresponding partial actions.
\end{remark}

Partial actions on modules arising from partial representations can be globalized.

\begin{theorem}[Theorem 2.33 \cite{Article_Jerez_2024_OTHOPGA}]\label{t: globalization of partial representations}
    Let $\pi: G \to \operatorname{End}_{K}(M)$ be a partial group representation, and let $\alpha: G \curvearrowright M$ be the partial group action induced by $\pi$. Then, $\Lambda(\alpha): G \curvearrowright \Lambda(M)$  is the universal globalization of $\alpha$.
\end{theorem}

We are interested on partial representations of groups into the endomorphism monoid of a non-associative algebra.
Let  \(A \in \mathcal{V}(I)\), and let \(\operatorname{End}_0(A)\) be the monoid of all algebra endomorphisms of \(A\) (if \(A\) possesses a unit element, we do not require the endomorphisms to preserve the unit element). Thus, we define:
\begin{definition}
    Let \(\pi: G \to \operatorname{End}_0(A)\) be a partial representation of \(G\) into \(\operatorname{End}_0(A)\), then we say that \(\pi\) is a \textbf{ partial representation of \(G\) on the algebra \(A\)}, if moreover, for all \(g \in G\) we have that \(\im \pi_g\) is a bilateral ideal of \(A\) we say that \(\pi\) is an {\bf ideal partial representation} of \(G\) on \(A\).
\end{definition}

Recall that a partial action $\alpha: G\curvearrowright A$ on a unital associative algebra $A$ is said to be unital if each ideal $A_g$ is  a unital algebra, i.e. the ideal $A_g$ is generated by an idempotent which is central in $A.$

\begin{proposition}\label{prop:one-to-one}
    Let \(A\) be a unital associative algebra.
    Then, there is a one-to-one correspondence between ideal partial representations  of \(G\) on \(A\) and unital partial actions of \(G\) on \(A\).
\end{proposition}
\begin{proof}
    Let \(\alpha: G \curvearrowright A\) be a unital partial action. For each $g\in G$, denote by $1_g$ the unity of the ideal  $A_g$. We can define an  ideal partial representation \(\pi: G \to \operatorname{End}_0(A)\) by setting \(\pi_g(a) = \alpha_g(1_{g^{-1}} a)\) for all \(g \in G\) and \(a \in A\).
    On the other  hand,  Remark~\ref{r: partial actions induced by partial representations} shows that the partial action on the $K$-module $A$ induced by an  ideal  partial  representation is a   partial action on the algebra $A$ such that $\pi _g (1)$ is the unity element of $\pi _g (A).$
\end{proof}

Let \( \pi: G \to \operatorname{End}_0(A) \) be a partial representation.  Note that for \( g \in G \) the element   \( \varepsilon_g = \pi_g \pi_{g^{-1}} \) is an idempotent and satisfies \( \varepsilon_g \pi_g = \pi_g \). Observe also that  conditions (ii) and (iii) in Definition~\ref{d: partial representation} can be rewritten as:
\[
    \pi_g \pi_h = \pi_{gh} \varepsilon_{h^{-1}} \quad \text{and} \quad \pi_g \pi_h = \varepsilon_g \pi_{gh}.
\]

If \( \pi \) is    ideal  partial representation, then for all \( a, b \in A \) and \( g, h \in G \), the following useful relation holds:
\begin{equation} \label{eq: relations of pi}
    \pi_g(a) \pi_h(b) = \pi_g(a \pi_{g^{-1}h}(b)).
\end{equation}
To verify this, observe that
\[
    \pi_g(a \pi_{g^{-1}h}(b)) = \pi_g(a) \pi_g(\pi_{g^{-1}h}(b)) = \varepsilon_g \pi_g(a) \varepsilon_g\pi_h(b) = \varepsilon_g(\pi_g(a) \pi_h(b)) = \pi_g(a) \pi_h(b).
\]

It is important to notice that if $\pi: G\to \operatorname{End}_K(M)$ is a partial representation, where $M$ is a $K$-module, then we can consider the mapping
\begin{align} \label{eq: def tau}
    \begin{split}
        \tau: \Lambda(M) &\to M \\
        \lfloor g, m \rfloor &\mapsto \pi_g(m) .
    \end{split}
\end{align}
Note that for $g,h\in G$ and $m\in \im \pi_{h^{-1}}$  we have 
$$\pi_g\pi_h(m)=\pi_{gh}\varepsilon_{h^{-1}}(m)=\pi_{gh}(m).$$
Applyig Lemma~\ref{l: computation tool for Lambda}, we  concluded that $\tau$ is a well-defined $K$-linear map. Moreover,  $\tau$ satisfies the relation $\tau\circ \iota=1_M$.

\section{Globalization of partial actions on varieties of algebras}

\subsection{Ideal partial representations}
In this section all partial representations are  assumed to be  ideal, unless stated otherwise.

\begin{definition}\label{d: globalization of partial group actions}
    Let \(A\) be an algebra in \(\mathcal{V}(I)\) and \(\alpha\) be a partial group action of \(G\) on  the algebra  \(A\).
    A \textbf{globalization} of \(\alpha\) is a pair \((\Theta, \iota)\), where \(\Theta: G \curvearrowright B\) is a global action on an algebra in \(\mathcal{V}(I)\) and \(\iota: A \to B\) is an injective \(G\)-equivariant map, such that:
    \begin{enumerate}[(i)]
        \item \(\iota(A)\) is a bilateral ideal of \(B\);
        \item \(\alpha\) is isomorphic to the restriction of \(\Theta\) to \(\iota(A)\);
        \item \(B = \sum_{g \in G} \Theta_g(\iota(A))\).
    \end{enumerate}
\end{definition}

In this section, we will show that if the partial group action \(\alpha\) on \(A\) arises from an  ideal partial representation  of \(G\) on \(A\), then there exists a globalization for \(\alpha\).
Moreover, the globalization that we will obtain is universal among all globalizations of \(\alpha\).

Let \( A \) be an algebra in \( \mathcal{V}(I) \), \( \pi: G \to \operatorname{End}_0(A) \)  an  ideal partial representation of \( G \) on \(A\), and \( \alpha \) the partial action on the algebra \( A \) induced by \( \pi \). Using the construction in~\cite{Article_Jerez_2024_OTHOPGA}, we  shall obtain a globalization for \( \alpha \). To this end, as in Subsection~\ref{Generalized partial group actions}, let \( \Lambda(A) \) be the quotient of the \( K \)-module \( KG \otimes A \) by the \( K \)-submodule generated by the set 
\[
    \{ g \otimes \alpha_h(a) - gh \otimes a : g, h \in G \text{ and } a \in A_{h^{-1}} \}.
\]

By Theorem~\ref{t: universal global action}, \( (\Lambda(\alpha), \iota) \) is a universal global action of the partial action of \( G \) on the \( K \)-module \( A \) induced by \( \pi \), where \( \Lambda(\alpha): G \curvearrowright \Lambda(A) \) and \( \iota: A \to \Lambda(A) \) are as in \eqref{eq: Theta} and \eqref{eq: iota equivariant map}, respectively. Throughout this section, for  notational convenience, we denote the action \( \Lambda(\alpha) \) by \( \Theta \).

\begin{lemma} \label{l: product}
    Let \( A \) be an algebra in \( \mathcal{V}(I) \), and \( \pi: G \to \operatorname{End}_0(A) \)  an ideal partial representation  of $G$ on \(A\). Then, the map 
    \[
        \cdot : \Lambda(A) \times \Lambda(A) \to \Lambda(A)
    \]
    defined by 
    \[
        \lfloor g, a \rfloor \cdot \lfloor h, b \rfloor := \lfloor g, a \pi_{g^{-1}h}(b) \rfloor
    \]
    is a well-defined \( K \)-bilinear map.
    Consequently, \(\Lambda(A)\) is a non-associative algebra.
    Moreover, the multiplication satisfies 
   \begin{equation} \label{eq: multiplication equivalence}
        \lfloor g, a \rfloor \cdot \lfloor h, b \rfloor
        = \lfloor h, \pi_{h^{-1}g}(a) b \rfloor,
    \end{equation} for all $g,h \in G,$ $a,b \in A.$
    
\end{lemma}

\begin{proof}
    First, we show that \( \cdot : \Lambda(A) \times \Lambda(A) \to \Lambda(A) \) is well-defined.    Consider the map \( \phi: (KG \otimes A)  \times (KG \otimes A) \to \Lambda(A) \) given  by \( \phi(g \otimes a, k \otimes c) := \lfloor g, a \pi_{g^{-1}k}(c) \rfloor \).
        Note that for every \( a \in A_{h^{-1}} \) we have
        \begin{align*}
            \phi( g \otimes \alpha_{h}(a) - gh \otimes a, k \otimes c)
            &= \lfloor g,\alpha_h(a)\pi_{g^{-1}k}(c)\rfloor-\lfloor gh,a\pi_{h^{-1}g^{-1}k}(c)\rfloor \\ 
            (\text{by }\eqref{eq: relations of pi})&= \lfloor g,\pi_h(a\pi_{h^{-1}g^{-1}k}(c))\rfloor-\lfloor gh,a\pi_{h^{-1}g^{-1}k}(c)\rfloor \\
            &=\lfloor gh,a\pi_{h^{-1}g^{-1}k}(c)\rfloor-\lfloor gh,a\pi_{h^{-1}g^{-1}k}(c)\rfloor=0.
        \end{align*}
        Hence, by Lemma~\ref{l: computation tool for Lambda}, there exists a unique linear map \( \phi_0 : \Lambda(A) \times  (KG \otimes A) \to \Lambda(
            A) \) such that  \(\phi_0( \lfloor g, a \rfloor , h \otimes b ) = \lfloor g, a \pi_{g^{-1}h}(b) \rfloor \).
        In the same fashion we see that
        \[
            \phi_0(\lfloor k, c \rfloor, g \otimes \alpha_{h}(a) - gh \otimes a) = 0,
        \]
        and therefore applying Lemma~\ref{l: computation tool for Lambda} again,  we conclude that the map \( \cdot : \Lambda(A) \times \Lambda(A) \to \Lambda(A) \) is well-defined.

    Now we prove the final assertion. For this we calculate
    \begin{align*}
        \lfloor g, a \pi_{g^{-1}h}(b) \rfloor 
        &= \lfloor g, \pi_{g^{-1}h} \pi_{h^{-1}g}(a \pi_{g^{-1}h}(b)) \rfloor  \\
        (\text{by }\eqref{eq: relations of pi})&=\lfloor g, \pi_{g^{-1}h}(\pi_{h^{-1}g}(a) b) \rfloor \\
        &= \lfloor h, \pi_{h^{-1}g}(a) b \rfloor,
    \end{align*}
    as required.
\end{proof}

As usual, we omit the dot in the notation for the  multiplication on \(\Lambda(A)\).

    \begin{lemma} \label{l: nonassociative monomials and Lambda}
    Let \(f(x_1, x_2, \ldots, x_n)\) be a non-associative monomial of degree \(n\), where each \(x_i\) appears exactly once in the order \(x_1, x_2, \ldots, x_n\).
    This means \(f\) corresponds uniquely to a parenthesization of the product \(x_1x_2\cdots x_n\).
    Then, for any elements  \( \lfloor g_1, a_1 \rfloor, \ldots, \lfloor g_n, a_n \rfloor\) in \(\Lambda(A)\), we have
    \[
        f(\lfloor g_1, a_1 \rfloor, \ldots, \lfloor g_n, a_n \rfloor) 
        = \lfloor g_i, f(\pi_{g_i^{-1}g_1}(a_1), \ldots, a_i, \ldots, \pi_{g_i^{-1}g_n}(a_n)) \rfloor,
    \]
    for all \(1 \leq i \leq n\).
\end{lemma}
\begin{proof}
    We proceed by induction on the number of variables.
    For \(n = 1\) the result is trivial, and for \(n=2\) it follows from equation~\eqref{eq: multiplication equivalence}.
     For \(n+1\), since $(x_j x_{j+1})$ appears in $f$ for some \(1 \leq j \leq n\), we can write \(f(x_1, \ldots, x_{n+1}) = f'(x_1, \ldots, x_j x_{j+1}, \ldots, x_{n+1})\), where  \(f'\) is a parenthesization monomial of degree \(n\).
    Let \(1 \leq i \leq n\),  then we consider the following cases:\\
   \textbf{Case 1:} Suppose  that \(i < j, j+1\), thus
    \begin{align*}
        f(\lfloor &g_1, a_1 \rfloor, \ldots, \lfloor g_{n+1}, a_{n+1} \rfloor) \\ 
        &= f'(\lfloor g_1, a_1 \rfloor, \ldots, \lfloor g_j, a_j \rfloor \lfloor g_{j+1}, a_{j+1} \rfloor, \ldots, \lfloor g_n, a_n \rfloor) \\
        &= f'(\lfloor g_1, a_1 \rfloor, \ldots,
            \lfloor g_j, a_{j}\pi_{g_j^{-1}g_{j+1}}(a_{j+1})\rfloor,
            \ldots, \lfloor g_n, a_n \rfloor) \\
        (\flat)&= \lfloor 
            g_i,
            f'(
            \pi_{g_i^{-1}g_1}(a_1), \ldots, a_i, \ldots,
            \pi_{g_i^{-1}g_j}(a_j \pi_{g_j^{-1}g_{j+1}}(a_{j+1})),
            \ldots,\pi_{g_i^{-1}g_n}(a_n)
            )
            \rfloor \\
        (\sharp)&= \lfloor 
            g_i,
            f'(
            \pi_{g_i^{-1}g_1}(a_1), \ldots, a_i, \ldots,
            \pi_{g_i^{-1}g_j}(a_j) \pi_{g_i^{-1}g_{j+1}}(a_{j+1}),
            \ldots,\pi_{g_i^{-1}g_n}(a_n)
            )
            \rfloor \\
        &= \lfloor 
            g_i,
            f(
            \pi_{g_i^{-1}g_1}(a_1), \ldots, a_i, \ldots,
            \pi_{g_i^{-1}g_j}(a_j), \pi_{g_i^{-1}g_{j+1}}(a_{j+1}),
            \ldots,\pi_{g_i^{-1}g_n}(a_n)
            )
            \rfloor,
    \end{align*}
    where \((\flat)\) holds by the induction hypothesis, and \((\sharp)\) follows from \eqref{eq: relations of pi}.\\
 \textbf{Case 2:} Assume that \(i > j, j+1\), then the proof is analogous to the previous case.\\
  \textbf{Case 3:} Suppose that \(i = j\), then
    \begin{align*}
        f(\lfloor &g_1, a_1 \rfloor, \ldots, \lfloor g_{n+1}, a_{n+1} \rfloor) \\ 
        &= f'(\lfloor g_1, a_1 \rfloor, \ldots, \lfloor g_j, a_j \rfloor \lfloor g_{j+1}, a_{j+1} \rfloor, \ldots, \lfloor g_n, a_n \rfloor) \\
        &= f'(\lfloor g_1, a_1 \rfloor, \ldots,
            \lfloor g_j, a_{j}\pi_{g_j^{-1}g_{j+1}}(a_{j+1})\rfloor,
            \ldots, \lfloor g_n, a_n \rfloor) \\
        (\flat)&= \lfloor 
            g_j,
            f'(
            \pi_{g_j^{-1}g_1}(a_1), a_j \pi_{g_j^{-1}g_{j+1}}(a_{j+1}),
            \ldots,\pi_{g_j^{-1}g_n}(a_n)
            )
            \rfloor \\
        &= \lfloor 
            g_j,
            f(
            \pi_{g_j^{-1}g_1}(a_1), a_j, \pi_{g_j^{-1}g_{j+1}}(a_{j+1}),
            \ldots,\pi_{g_j^{-1}g_n}(a_n)
            )
            \rfloor,
    \end{align*}
    where \((\flat)\) holds by the induction hypothesis.

    \textbf{Case 4:} Assume that \(i = j+1\). The proof follows similarly to the previous case.
\end{proof}

\begin{proposition}\label{p: nonassociative monomials and Lambda}
    Let \(f(x_{1}, \ldots, x_{n}) \neq 0\) be a non-associative monomial such that each \(x_{i}\) appears in \(f(x_{1}, \ldots, x_{n})\). 
    Then, for any \( \lfloor g_1, a_1 \rfloor, \ldots, \lfloor g_n, a_n \rfloor \in \Lambda(A) \setminus \{ 0 \}\), we have
    \begin{equation}\label{eq:nonassoc monomial and Lambda}
        f(\lfloor g_1, a_1 \rfloor, \ldots, \lfloor g_n, a_n \rfloor) 
        = \lfloor g_i, f(\pi_{g_i^{-1}g_1}(a_1), \ldots, a_i, \ldots, \pi_{g_i^{-1}g_n}(a_n)) \rfloor.
    \end{equation}
\end{proposition}
\begin{proof}
    First note that there exists a non-associative monomial \(f'(y_1, \ldots, y_m)\) as in  the statement of Lemma~\ref{l: nonassociative monomials and Lambda} and \(\sigma \in \{ 1, \ldots, n \}^{m}\),  $m\geq n,$ such that \(f'(x_{\sigma_1}, \ldots, x_{\sigma_m}) = f(x_1, \ldots, x_n)\).
    Let \(\lfloor g_1, a_1 \rfloor, \ldots, \lfloor g_n, a_n \rfloor \in \Lambda(A)\),
    and \( 1 \leq i \leq n\), then
    \begin{align*}
        f(\lfloor g_1, a_1 \rfloor, \ldots, \lfloor g_n, a_n \rfloor)
        &= f'(\lfloor g_{\sigma_1}, a_{\sigma_{1}}\rfloor, \ldots, \lfloor g_{\sigma_{m}}, a_{\sigma_{m}} \rfloor ) \\
        (\flat) &= \lfloor g_i, f'(\pi_{g_i^{-1}g_{\sigma_1}}(a_{\sigma_1}), \ldots, \pi_{g_i^{-1}g_{\sigma_m}}(a_{\sigma_m})) \rfloor \\
        &= \lfloor g_i, f(\pi_{g_i^{-1}g_{1}}(a_1), \ldots, \pi_{g_i^{-1}g_{i}}(a_i), \ldots,\pi_{g_i^{-1}g_m}(a_m)) \rfloor \\
        &= \lfloor g_i, f(\pi_{g_i^{-1}g_{1}}(a_1), \ldots, a_i, \ldots,\pi_{g_i^{-1}g_m}(a_m)) \rfloor,
    \end{align*}
    where \((\flat)\) holds by Lemma~\ref{l: nonassociative monomials and Lambda}.
\end{proof}

\begin{remark}\label{rem:nonassoc monomial and Lambda}  For the proof of the theorem below observe that  it easily follows from Proposition~\ref{p: nonassociative monomials and Lambda} that if \(f(x_{1}, \ldots, x_{n}) \neq 0\) is a non-associative monomial such that   \(x_{i}\) appears in \(f(x_{1}, \ldots, x_{n})\) for some $i$,  
    then, for any \( \lfloor g_1, a_1 \rfloor, \ldots, \lfloor g_n, a_n \rfloor \in \Lambda(A) \setminus \{ 0 \}\), 
    the equality \eqref{eq:nonassoc monomial and Lambda} holds for $i.$ For it is enough to consider   $f$  as a monomial  only on those variables which indeed appear in $f$ and apply Proposition~\ref{p: nonassociative monomials and Lambda}.
\end{remark}

\begin{theorem}\label{t: globalization on varieties}
    Let $A$ be an algebra in a variety $\mathcal{V}(I)$ and $\pi: G \to \operatorname{End}_0(A)$ be   an ideal partial representation of $G$ on \(A\).
    Suppose that for any $f(x_1, \ldots, x_n) \in I$, there exists \(i\) such that \(x_i\) appears in each monomial of $f$. Then $(\Theta, \iota)$ is a globalization for the partial action $\alpha$  on the algebra $A$ induced by $\pi$ such that \(\Lambda(A) \in \mathcal{V}(I)\).
    Moreover, the product defined in~\eqref{eq: multiplication equivalence} is the unique one such that $\tau: \Lambda(A) \to A$, defined by $\tau(\lfloor g, a \rfloor) = \pi_g(a)$ (as in~\eqref{eq: def tau}), is an algebra homomorphism, $\iota(A)$ is an ideal of $\Lambda(A)$, and $(\Theta, \iota)$ is a globalization for $\alpha$.
\end{theorem}

\begin{proof}
   By Theorem~\ref{t: globalization of partial representations}, $(\Theta, \iota)$ is the universal globalization action of the partial action on the $K$-module $A$ induced by $\pi$. We already know that $\Lambda(A)$ is a non-associative algebra by Lemma~\ref{l: product}. Thus, for the first statement  it is sufficient to show that \(\Lambda(A) \in \mathcal{V}(I)\) and that the map \(\iota: A \to \Lambda(A)\) is an algebra morphism such that \(\iota(A)\) is a bilateral ideal of \(\Lambda(A)\).  We now show that $\Lambda(A)$, with the product defined in~\eqref{eq: multiplication equivalence}, belongs to $\mathcal{V}(I)$. Let $f(x_1, \ldots, x_n) \in I$ and suppose that there exists \(1 \leq i \leq n\)  such that $x_i$ appears in every monomial of $f$.
    Let \(\operatorname{M}_f\) be the set of all monomials,  which are summands of $f$,
    and let $\lfloor g_1, a_1 \rfloor, \ldots, \lfloor g_n, a_n \rfloor \in \Lambda(A)$, then  by Proposition~\ref{p: nonassociative monomials and Lambda} (see also Remark~\ref{rem:nonassoc monomial and Lambda}) we have that
    \begin{align*}
        f(\lfloor g_1, a_1 \rfloor, \ldots, \lfloor g_n, a_n \rfloor)
        &= \sum_{t \in \operatorname{M}_f} t( \lfloor g_1, a_1 \rfloor, \ldots, \lfloor g_n, a_n \rfloor) \\
        &= \sum_{t \in \operatorname{M}_f} \lfloor g_i, t(\pi_{g_i^{-1}g_1}(a_1), \ldots, a_i, \ldots, \pi_{g_i^{-1}g_n}(a_n)) \rfloor \\
        &= \lfloor g_i, f(\pi_{g_i^{-1}g_1}(a_1), \ldots, a_i, \ldots, \pi_{g_i^{-1}g_n}(a_n)) \rfloor = 0,
    \end{align*}
which implies that $\Lambda(A) \in \mathcal{V}(I)$. 
    Now, note that
    \[
        \iota(a)\iota(b) = \lfloor 1, a \rfloor \lfloor 1, b \rfloor = \lfloor 1, a \pi_{1}(b) \rfloor = \lfloor 1, ab \rfloor = \iota(ab),
    \]
    hence $\iota$ is an algebra homomorphism, and $\iota(A)$ is a bilateral ideal of $\Lambda(A)$ by equation \eqref{eq: multiplication equivalence}.

    Next we prove the assertion about $\tau: \Lambda(A) \to A$. Let $\lfloor g, a \rfloor, \lfloor h, b \rfloor \in \Lambda(A)$. Then,
    \[
       \tau(\lfloor g, a \rfloor \lfloor h, b \rfloor) = \tau(\lfloor g, a \pi_{g^{-1}h}(b) \rfloor)= \pi_g(a \pi_{g^{-1}h}(b))\overset{\scriptsize\eqref{eq: relations of pi}}{=}\pi_g(a) \pi_h(b) = \tau(\lfloor g, a \rfloor) \tau(\lfloor h, b \rfloor),
    \]
    so $\tau$ is an algebra homomorphism. 

    Finally, assume that $*$ is a product in $\Lambda(A)$ such that $\tau$ is an algebra homomorphism, $\iota(A) \unlhd \Lambda(A)$, and $(\Theta, \iota)$ is a globalization for $\alpha$. For any $a \in A$ and $\lfloor h, b \rfloor \in \Lambda(A)$, we have $\lfloor 1, a \rfloor * \lfloor h, b \rfloor = \lfloor 1, c \rfloor$ for some $c \in A$. Thus,
    $$ c = \tau(\lfloor 1, c \rfloor) = \tau(\lfloor 1, a \rfloor * \lfloor h, b \rfloor) = \tau(\lfloor 1, a \rfloor) \tau(\lfloor h, b \rfloor) = a \pi_h(b), $$
    which gives
    $$ \lfloor 1, a \rfloor * \lfloor h, b \rfloor = \lfloor 1, c \rfloor = \lfloor 1, a \pi_h(b) \rfloor = \lfloor 1, a \rfloor \lfloor h, b \rfloor. $$ Hence, for any $\lfloor g, a \rfloor, \lfloor h, b \rfloor \in \Lambda(A)$, we have
    \begin{multline*}
        \lfloor g, a \rfloor * \lfloor h, b \rfloor = \Theta_g(\lfloor 1, a \rfloor) * \Theta_g(\lfloor g^{-1}h, b \rfloor) = \Theta_g(\lfloor 1, a \rfloor * \lfloor g^{-1}h, b \rfloor) = \\
        = \Theta_g(\lfloor 1, a \pi_{g^{-1}h}(b) \rfloor) = \lfloor g, a \pi_{g^{-1}h}(b) \rfloor = \lfloor g, a \rfloor \lfloor h, b \rfloor,
    \end{multline*}
    as required.
\end{proof}

\begin{remark}
    There are many interesting examples of varieties $\mathcal{V}(I)$ in which $I$ satisfies the hypotheses of Theorem~\ref{t: globalization on varieties}.  Clearly, $I$ could be any set of homogeneous non-associative polynomials. For instance, this is the case of the varieties of associative algebras, Lie algebras, Jordan algebras, noncommutative Jordan algebras, alternative algebras, Malcev algebras,  Leibniz algebras and  power associative algebras,  among others.  As we mentioned already, any variety of non-associative algebras over an infinite field is homogeneous and, consequently, it is defined by a set $I,$ consisting of homogeneous non-associative polynomials, so that it fits Theorem~\ref{t: globalization on varieties}.
\end{remark}

\begin{remark}\label{r: about universal condition}
    Suppose that our algebra \(A\) is idempotent.
    Moreover, suppose that another algebra \(B \in \mathcal{V}(I)\) contains \(A\) as an ideal, and we have a global action \(\Theta': G \curvearrowright B\) such that \(\alpha\) is the restriction of \(\Theta'\) to \(A\). 
    Then,
    \begin{equation}\label{eq: ideal universal condition}
        \Theta'_g(a)b = \Theta'_{g}(a)\varepsilon_g(b), \text{ for all }a,b \in A.
    \end{equation}
    Indeed, it is enough to suppose that \(b= b_1 b_2\) for some \(b_1,b_2 \in A\), then
    \begin{align*}
        \Theta'_g(a)b
        &= \Theta'_g(a)b_1 b_2
        = \varepsilon_{g}(\Theta'_{g}(a) b_1 b_2)
        = \varepsilon_{g}(\Theta'_{g}(a)b_1 \varepsilon_{g}(b_2)) \\
        &= \Theta'_{g}(a) \varepsilon_{g}(b_1 \varepsilon_{g}(b_{2}))
        = \Theta'_{g}(a) \varepsilon_{g}(b_1b_{2})
        = \Theta'_{g}(a) \varepsilon_{g}(b).
    \end{align*}
\end{remark}

Note that condition \eqref{eq: ideal universal condition} holds beyond the idempotent algebras, in particular \(\Theta: G \curvearrowright \Lambda(A)\) satisfies that equation.
Furthermore, we can show that the globalization obtained in Theorem~\ref{t: globalization on varieties} is universal among all the globalizations satisfying \eqref{eq: ideal universal condition}.

\begin{lemma} \label{l: basic computation for globalizations}
    Let \(A\) be an algebra in \(\mathcal{V}(I)\), \(\pi\)  an ideal  partial representation  of \(G\) on \(A\), and let \(\alpha: G \curvearrowright A\) be the partial action induced by \(\pi\).  
    Then, for any algebra \(B\) in \(\mathcal{V}(I)\), and every  global action \(\Theta': G \curvearrowright B\)  on \(B\) such that \(A\) is an ideal of \(B\), \(\alpha\) is the restriction of \(\Theta'\) to \(A\),
    and \eqref{eq: ideal universal condition} holds.
    Then,
    \[
        \Theta'_g(a)b = \Theta'_g(a \pi_{g^{-1}}(b)) = \alpha_g(a \pi_{g^{-1}}(b)),
    \]
    for all \(a,b \in A\).
\end{lemma}
\begin{proof}
    First note that \(\Theta'_g(a)b \in \Theta'_g(A)\cap A= A_g\).
    Hence,
    \begin{align*}
        \Theta'_{g}(a)b 
        &= \Theta'_{g}(a) \varepsilon_{g} (b)
        = \Theta'_{g}(a) \alpha_{g} \pi_{g^{-1}} (b) \\
        &= \Theta'_{g}(a) \Theta'_{g} \pi_{g^{-1}} (b)
        = \Theta'_{g}(a \pi_{g^{-1}} (b))
        = \alpha_{g}(a \pi_{g^{-1}} (b)).
    \end{align*}
\end{proof}

\begin{proposition}
    Let  \(A\in \mathcal{V}(I)\), \(\alpha: G \curvearrowright A \) be the partial action induced by  an ideal partial representation \(\pi: G\to \operatorname{End}_0(A)\), and \(\Theta: G \curvearrowright \Lambda(A) \)  the globalization obtained in Theorem~\ref{t: globalization on varieties}.
    If \(\Theta': G \curvearrowright B \), \(B\in \mathcal{V}(I)\), is a globalization of \(\alpha\) such that \eqref{eq: ideal universal condition} holds. Then, there exists a surjective  $G$-equivariant algebra homomorphism \(\varphi: \Lambda(A) \to B \) such that \(\varphi \circ \iota = \iota'\), where \(\iota: A \to \Lambda(A) \) and \(\iota': A \to B \) are the embeddings of \(A\) in \(\Lambda(A) \) and \(B\), respectively.
\end{proposition}
\begin{proof}
    Note that  by Theorem~\ref{t: globalization of partial representations}, \(\Theta\) is the universal globalization of \(\alpha\) seen as partial action on the module structure of \(A\).
     Therefore, there exists   a $G$-equivariant surjective linear map \(\varphi: \Lambda(A) \to B \) such that \(\varphi \circ \iota = \iota'\).  Notice that
    \[
        \varphi(\lfloor g, a \rfloor) = \varphi(\Theta_g(\iota(a))) = \Theta'_g(\varphi(\iota(a)))  = \Theta'_g(\iota'(a))
    \] for all $g\in G$ and $a\in A.$
    
      In order to verify that \(\varphi\) is an algebra homomorphism
    observe that for any \(g \in G\), \(a, b \in A\) we have
    \begin{align*}
        \varphi(\lfloor g, a \rfloor \lfloor 1, b \rfloor)
        &= \varphi( \lfloor g, a \pi_{g^{-1}}(b) \rfloor ) 
        =  \Theta_g' \iota'(a \pi_{g^{-1}}(b))
        =  \Theta_g' (\iota'(a) \iota'(\pi_{g^{-1}}(b) ) ) \\
        (\text{by Lemma~\ref{l: basic computation for globalizations}})
        &= \Theta_g' (\iota'(a)) \iota'(b) 
        = \varphi(\lfloor g, a \rfloor)  \varphi(\lfloor 1, b \rfloor). 
    \end{align*}  Then
    \begin{align*} 
     \varphi(\lfloor g, a \rfloor \lfloor h, b \rfloor)
        &= \varphi(\lfloor g, a \rfloor \Theta_h( \lfloor 1, b \rfloor))
         = \varphi(\Theta_h(\lfloor h^{-1}g, a \rfloor  \lfloor 1, b \rfloor))
         = \Theta '_h(\varphi(\lfloor h^{-1}g, a \rfloor  \lfloor 1, b \rfloor))\\
         &= \Theta '_h(\varphi(\lfloor h^{-1}g, a \rfloor ) \varphi(\lfloor 1, b \rfloor))
         = \varphi(\lfloor g, a \rfloor) \varphi( \lfloor h, b \rfloor),
    \end{align*} for any  $g,h\in G, a,b\in A.$
    
\end{proof}

\subsection{Globalization of generalized partial actions}

One of the main motivations for this section is to identify a suitable framework in which generalized partial actions can be globalized. The example below illustrates a situation where generalized partial actions arise naturally.
\begin{example}\label{e: partial action on U(L)}
    Let $\pi: G \to \operatorname{End}_0(L)$ be  an ideal partial representation of $G$ on a Lie algebra $L$, and let $\alpha$ denote the partial action induced by $\pi$.
    By the universal property of the universal enveloping algebra $U(L)$ of $L$, $\pi$ induces a partial representation (but not  an ideal one) $U(\pi) : G \to \operatorname{End}_0(U(L))$ of $G$ on $U(L)$ and \(U(\pi)\) gives rise to a generalized partial action $U(\alpha)$ on $U(L)$. 
     Observe that each \(U(\pi)_g\), being an endomorphism of the associative algebra $U(L)$, is also an endomorphism of the corresponding Lie algebra  $U(L)^-$, so that $U(\alpha)$  is also a generalized partial action of $G$ on the Lie algebra  $U(L)^-,$ and the canonical map $L\to U(L)$ is Lie $G$-equivariant.  In particular, it is a 
    $G$-equivariant linear map.
    
    On the other hand, \(U(\Lambda(L))\) possesses a natural global \(G\)-action \(U(\Theta)\) induced by the global action of \(G\) on \(\Lambda(L)\)
     and, as above,  the canonical map 
    \(\Lambda(L) \to U(\Lambda(L)) \) is  \(G\)-equivariant. Since  \( \iota : L \to  \Lambda (L)   \)  is a \(G\)-equivariant  Lie algebra embedding, it induces a \(G\)-equivariant map  \(U(\iota ): U(L) \to  U (\Lambda(L))\).
    Therefore, we have the following commutative diagram, in which all maps are  \(G\)-equivariant:
    \begin{center}
    \begin{tikzcd}
        L & {U(L)} \\
        {\Lambda(L)} & U(\Lambda(L))
        \arrow[from=1-1, to=1-2]
        \arrow["\iota"', from=1-1, to=2-1]
        \arrow["{U(\iota)}", from=1-2, to=2-2]
        \arrow[from=2-1, to=2-2]
    \end{tikzcd}
\end{center}
    One can verify that \((U(\Theta), U(\iota))\) is in fact a globalization of \(U(\alpha)\)
    in the sense of Definition~\ref{def:GlobalizationForGeneralizewdParActions} below.
    Thus, it gives a construction to globalize generalized partial actions.
    We  will come back to this example in Example~\ref{e: partial action on U(L) 2}, once we have the necessary tools to prove the existence of globalizations for generalized partial actions.
\end{example}

At this point, we address the existence of a universal global action and a globalization for a given generalized partial action \(\alpha: G \curvearrowright A\) of \(G\) on an algebra \(A \in \mathcal{V}(I)\).

\begin{definition}\label{def:GlobalizationForGeneralizewdParActions}
    Let \(\alpha: G \curvearrowright A\) be a generalized partial action of \(G\) on the algebra \(A \in \mathcal{V}(I)\). We define:
    \begin{enumerate}[(i)]
        \item The \textbf{universal global action} of \(\alpha\) to be a pair \((\Theta, \iota)\), where \(\Theta: G \curvearrowright B\) is a global action on an algebra \(B \in \mathcal{V}(I)\) and \(\iota: A \to B\) is a \(G\)-equivariant map, such that for any global action \(\beta: G \curvearrowright X\) on an algebra \(X \in \mathcal{V}(I)\) and any \(G\)-equivariant map \(\psi: A \to X\), there exists a unique $G$-equivariant map \(\widetilde{\psi}: B \to X\) such that \(\widetilde{\psi} \circ \iota = \psi\).
        \item A \textbf{globalization} of \(\alpha\) to be a pair \((\theta, \nu)\), where \(\theta: G \curvearrowright B\) is a global action on an algebra \(B \in \mathcal{V}(I)\) and \(\nu: A \to B\) is a \(G\)-equivariant map, such that \(\alpha\) and \(\theta|_{\im \nu}\) are isomorphic as generalized partial actions via \(\nu\) and \(\bigcup_{g \in G}\theta_{g}(\im \nu)\) generates \(B\) as an algebra.
    \end{enumerate}
\end{definition}

We are going to show that the universal global action of a generalized partial group action on a non-associative algebra always exists. Moreover, if we assume that the partial action arises from a partial representation, we shall prove that the universal global action is a globalization of the partial action. 

 Let \( A \in \mathcal{V}(I) \),  and denote by \((\Theta, \iota)\)  the universal global action of the partial action \(\alpha: G \curvearrowright A\) on the \(K\)-module \(A\), as described in Theorem~\ref{t: universal global action}. Now, consider \(F_I(\Lambda(A))\), the free algebra generated by the \(K\)-module \(\Lambda(A)\) in \(\mathcal{V}(I)\), together with the  canonical linear map \(j: \Lambda(A) \to F_I(\Lambda(A))\). Then by the universal property of \(F_I(\Lambda(A))\), there exists a unique global action \(\widetilde{\Theta}: G \curvearrowright F_I(\Lambda(A))\)  on the algebra \(F_I(\Lambda(A))\) such that \(\widetilde{\Theta}_g \circ j = j \circ \Theta_g\). That is, \(j\) is a $G$-equivariant map.

Let \(J\) denote the ideal of \(F_I(\Lambda(A))\) generated by the set 
\[
    \{j(\lfloor g, a \rfloor) j(\lfloor g, b \rfloor) - j(\lfloor g, ab \rfloor) \mid a, b \in A, \, g \in G\},
\]
and define 
\begin{equation}\label{eq def omega}
    \Omega(A) = F_I(\Lambda(A)) / J.
\end{equation}
 Obviously, $\Omega(A) \in \mathcal{V}(I) .$
Let \(p: F_I(\Lambda(A)) \to \Omega(A)\) be the canonical projection. It is straightforward to verify that \(J\) is a \(\widetilde{\Theta}\)-invariant ideal. Consequently, we obtain a global action 
\begin{equation}\label{eq def theta}
    \theta: G \curvearrowright \Omega(A),
\end{equation}
 on the algebra $\Omega(A)$ satisfying \(p \circ \widetilde{\Theta}_g = \theta_g \circ p\). In other words, \(p\) is a $G$-equivariant map.

Additionally, consider the map 
\begin{equation}\label{eq def nu}
    \nu := p \circ j \circ \iota: A \to \Omega(A).
\end{equation}
Since \(\nu\) is a composition of \(G\)-equivariant maps, it is \(G\)-equivariant too. Also, for any \(a, b \in A\), we have 
\[
    \nu(ab) = p(j(\lfloor 1, ab \rfloor)) = p(j(\lfloor 1, a \rfloor) j(\lfloor 1, b \rfloor)) = p(j(\lfloor 1, a \rfloor)) p(j(\lfloor 1, b \rfloor)) = \nu(a) \nu(b),
\]
showing that \(\nu\) is an algebra homomorphism.

\begin{remark}\label{r: generators of Omega A}
    For any generator $p(j(\lfloor g,a\rfloor))$ of $\Omega(A)$ (as an algebra), we have 
    \[
        p(j(\lfloor g,a\rfloor))=p\circ j\circ \Theta_g\circ \iota(a)=p\circ \widetilde\Theta_g\circ j\circ \iota(a)=\theta_g\circ p\circ j\circ \iota(a)=\theta_g(\nu(a)),
    \]
which implies that \(\Omega(A)\) is generated as an algebra by \(\bigcup_{g \in G} \theta_g(\nu(A))\).
       
\end{remark}

\begin{theorem}\label{p: universal global action in varieties}
    Let \( A \in \mathcal{V}(I) \), \(\alpha: G \curvearrowright A\) be a generalized partial action of \(G\) on the algebra \(A\), and let \(\theta: G \curvearrowright \Omega(A)\) and \(\nu: A \to \Omega(A)\) be as in~\eqref{eq def theta} and~\eqref{eq def nu}, respectively. Then, \((\theta, \nu)\) is a universal global action for \(\alpha\).
\end{theorem}

\begin{proof}
    Let \(\beta: G \curvearrowright X\) be a global action on the algebra \(X \in \mathcal{V}(I)\), and let \(\psi: A \to X\) be a \(G\)-equivariant morphism.
    Since \((\Theta, \iota)\) is a universal global action for the partial action \(\alpha: G \curvearrowright A\) on the \(K\)-module \(A\), there exists a unique \(K\)-linear \(G\)-equivariant map \(\widetilde{\psi}: \Lambda(A) \to X\), given by 
    \[
        \widetilde{\psi}(\lfloor g, a \rfloor) = \beta_g(\psi(a)),
    \]
    such that \(\widetilde{\psi} \circ \iota = \psi\).
    By the universal property of \(F_I(\Lambda(A))\), there exists a unique algebra homomorphism \(\overline{\psi}: F_I(\Lambda(A)) \to X\) such that \(\overline{\psi} \circ j = \widetilde{\psi}\).
    We  verify that \(\overline{\psi}\) is \(G\)-equivariant.
    For any \(\lfloor g, a \rfloor \in \Lambda(A)\) and \(h \in G\), we compute:
    \[
        \overline{\psi}(\widetilde{\Theta}_h(j(\lfloor g, a \rfloor))) 
        = \overline{\psi}(j(\Theta_h(\lfloor g, a \rfloor))) 
        = \overline{\psi}(j(\lfloor hg, a \rfloor)) 
        = \widetilde{\psi}(\lfloor hg, a \rfloor),
    \]
    and by the definition of \(\widetilde{\psi}\),
    \[
        \widetilde{\psi}(\lfloor hg, a \rfloor) = \beta_{hg}(\psi(a)) = \beta_h(\beta_g(\psi(a))) = \beta_h(\widetilde{\psi}(\lfloor g, a \rfloor)) = \beta_h(\overline{\psi}(j(\lfloor g, a \rfloor))).
    \]
    Thus, \(\overline{\psi}\) is \(G\)-equivariant.

    Next, we observe that \(J \subseteq \ker(\overline{\psi})\). Indeed, for \(a, b \in A\) and \(g \in G\), we calculate:
    \begin{align*}
        \overline{\psi}(j(\lfloor g, a \rfloor)j(\lfloor g, b \rfloor) - j(\lfloor g, ab \rfloor)) 
        &= \widetilde{\psi}(\lfloor g, a \rfloor)\widetilde{\psi}(\lfloor g, b \rfloor) - \widetilde{\psi}(\Theta_g(\lfloor 1, ab \rfloor)) \\
        &= \widetilde{\psi}(\Theta_g(\iota(a)))\widetilde{\psi}(\Theta_g(\iota(b))) - \widetilde{\psi}(\Theta_g(\iota(ab))) \\
        &= \beta_g(\psi(a))\beta_g(\psi(b)) - \beta_g(\psi(ab)) = 0.
    \end{align*}
    Thus, there exists a unique algebra homomorphism \(\widehat{\psi}: \Omega(A) \to X\) such that \(\widehat{\psi} \circ p = \overline{\psi}\). Consequently, \(\widehat{\psi} \circ \nu = \psi\).
    Since  \(\overline{\psi}\)  is  \(G\)-equivariant, it follows that \(\widehat{\psi}\) is also \(G\)-equivariant.  Hence, \((\theta, \nu)\) is  universal global action of \(\alpha\).
\end{proof}

\begin{corollary}\label{c: universal globalization}
    Let \( A \in \mathcal{V}(I) \), \(\alpha: G \curvearrowright A\) be a generalized partial action of \(G\) on the algebra \(A\), and let \(\theta: G \curvearrowright \Omega(A)\) and \(\nu: A \to \Omega(A)\) be as in~\eqref{eq def theta} and~\eqref{eq def nu}, respectively. If \(\alpha\) is globalizable, then \((\theta, \nu)\) is a globalization for \(\alpha\).
\end{corollary}

\begin{proof}
    Let \((\beta, \psi)\) be a globalization for \(\alpha\), where \(\beta: G \curvearrowright X\) is a global action on   the algebra \(X \in \mathcal{V}(I)\). By the universal property of \((\theta, \nu)\), there exists a unique \(G\)-equivariant map \(\widetilde{\psi}: \Omega(A) \to X\) such that \(\widetilde{\psi} \circ \nu = \psi\). Since \(\psi\) is injective, it follows that \(\nu\) is injective as well.
    By Remark~\ref{r: generators of Omega A}, to show that $(\theta,\nu)$ is a globalization for $\alpha$  it only remains to verify that \(\nu(A_g) = \theta_g(\nu(A)) \cap \nu(A)\), for any $g\in G$. Since $\nu$ is $G$-equivariant, the inclusion \(\nu(A_g)\subseteq \theta_g(\nu(A)) \cap \nu(A)\) is immediate. For the other inclusion, let \(a,b\in A\) such that \(\theta_g(\nu(a))=\nu(b)\).
    Then
    \[
        \psi(b)=\widetilde \psi(\nu(b))=\widetilde \psi(\theta_g(\nu(a)))=\beta_g(\psi(a))\in \psi(A)\cap \beta_g(\psi(A))=\psi(A_g),
    \]
  whence we have that $b\in A_g$, and hence \(\theta_g(\nu(A))\cap \nu(A)\subseteq \nu(A_g)\).
\end{proof}

\begin{corollary}\label{c: universal globalization 1}
    Let \( A \in \mathcal{V}(I) \) and \( \pi: G \to \operatorname{End}_0(A) \) be a  (non-necessarily ideal)  partial representation of \( G \) on \(A\). Then, \((\theta, \nu)\) is a globalization for the generalized partial action \(\alpha: G \curvearrowright A\) induced by \(\pi\).
\end{corollary}

\begin{proof}
    Recall that the \(K\)-linear map \(\tau: \Lambda(A) \to A\), defined by \(\tau(\lfloor g, a \rfloor) = \pi_g(a)\), satisfies \(\tau \circ \iota = \mathrm{id}_A\), where \(\iota: A \to \Lambda(A)\) is given by \(\iota(a) = \lfloor 1, a \rfloor\). By the universal property of \(F_I(\Lambda(A))\), there exists a unique algebra homomorphism \(\widetilde{\tau}: F_I(\Lambda(A)) \to A\) such that \(\widetilde{\tau} \circ j = \tau\). Now, since \(\pi_g: A \to A\) is an algebra homomorphism, it is straightforward to verify that the ideal \(J\), generated by the set 
    \[
        \{j(\lfloor g, a \rfloor)j(\lfloor g, b \rfloor) - j(\lfloor g, ab \rfloor) \mid a, b \in A, \, g \in G\},
    \]
    is contained in \(\ker \widetilde{\tau}\). Consequently, there exists a unique algebra homomorphism \(\widehat{\tau}: \Omega(A) \to A\) such that \(\widehat{\tau} \circ p = \widetilde{\tau}\). Hence \(\widehat{\tau} \circ \nu = \mathrm{id}_A\), which implies that \(\nu\) is an injective \(G\)-equivariant map.
     Clearly,  \(\nu(A_g)\subseteq \theta_g(\nu(A)) \cap \nu(A)\) for all $g\in G$.
    For the converse inclusion, let \(a,b \in A\) and \(g \in G\) be such that \(\theta_g(\nu(a))=\nu(b)\).
    Then we have that 
    \[
        \theta_g(\nu(a))
        = \theta_g \circ p \circ  j \circ \iota (a)
        = p \circ \widetilde{\Theta}_g \circ j\circ \iota (a)
        =  p \circ  j\circ \Theta _g \circ \iota (a)
        = p (  j (\lfloor g, a \rfloor )) 
    \]
    equals
    \(\nu (b) = p (  j ( \lfloor 1, b \rfloor )),\) 
    which implies that 
    \[
      j (\lfloor g, a \rfloor ) -  j ( \lfloor 1, b \rfloor ) \in J.
    \]
    Applying $\widetilde{\tau},$ we obtain 
    \[
      0=\tau (\lfloor g, a \rfloor )   -   \tau ( \lfloor 1, b \rfloor ) = \pi _g (a) - \pi _1 (b) =  \pi _g (a) - b,
    \]
    so that  $b = \pi _g (a) \in  A_g,$ which shows that  \( \theta_g(\nu(A)) \cap \nu(A) \subseteq \nu(A_g). \)  Consequently,
   \(\nu(A_g) = \theta_g(\nu(A)) \cap \nu(A)\), and Remark~\ref{r: generators of Omega A} implies that $(\theta,\nu)$ is a globalization for $\alpha$.
\end{proof}

\begin{example}\label{e: partial action on U(L) 2}
    Now we return to Example~\ref{e: partial action on U(L)}.
    If we include the algebras and morphisms given by \(\Omega\) to the diagram in Example~\ref{e: partial action on U(L)},
    then we obtain the following commutative diagram, in which all maps are \(G\)-equivariant:
    \[\begin{tikzcd}[sep=2.25em]
        {\Omega(L)} & {U(\Omega(L))} \\
        L & {U(L)} & {\Omega(U(L))} \\
        {\Lambda(L)} & {U(\Lambda(L))}
        \arrow[from=1-1, to=1-2]
        \arrow["\psi", curve={height=-6pt}, dashed, from=1-2, to=2-3]
        \arrow["{\nu_{L}}", from=2-1, to=1-1]
        \arrow[from=2-1, to=2-2]
        \arrow["\iota"', from=2-1, to=3-1]
        \arrow["{U(\nu_{L})}", from=2-2, to=1-2]
        \arrow["{\nu_{U(L)}}"', from=2-2, to=2-3]
        \arrow["{U(\iota)}"', from=2-2, to=3-2]
        \arrow["\varphi", curve={height=-6pt}, dashed, from=2-3, to=1-2]
        \arrow["\phi", curve={height=-6pt}, dashed, from=2-3, to=3-2]
        \arrow[from=3-1, to=3-2]
    \end{tikzcd}\]
    The maps $\phi$ and $\varphi$ are induced by the universal property of $\Omega(U(L))$.
    The map $\psi$ is induced by the universal property of $U(\Omega(L))$ and by the fact that the composition $L\to U(L)\to \Omega(U(L))$ induces a Lie algebra homomorphism $\Omega(L)\to\Omega(U(L))$.
    Finally, because both $\psi$ and $\varphi$ are induced by universal properties, a standard argument shows that they are inverse isomorphisms.
\end{example}

\section{Covariant representations of partial representations}


This section is devoted to the covariant representation on algebras. For convenience, we use the symbol $\cdot$ to denote the product on an algebra.     Also, all algebras will be associative (with or without unit) or Lie algebras over \(K\) unless otherwise stated.

 We recall that if $A$ is a Lie algebra, then a $K$-module $M$ is an $A$-module if there exists a $K$-bilinear map 
\(
A \times M \to M, (x, m) \mapsto x m,
\)
satisfying the condition 
\begin{equation} \label{eq:module_condition}
    (x \cdot y) m = x (y  m) - y (x m) \quad \text{for all } x, y \in L, \; m \in M.
\end{equation}
Equivalently, $M$ is an $A$-module if there exists a representation (Lie algebra homomorphism) 
\[
    \varphi: A \to \mathfrak{gl}(M),
\]
where $\mathfrak{gl}(M)$ denotes the Lie algebra of $K$-linear endomorphisms of $M$.

    Given an algebra \(A\) and a (global) representation \(\rho: G \to  \operatorname{End}_0 (A)\), by a covariant representation of \(\rho\) we mean a pair \((\pi, M)\), where \(M\) is an \(A\)-module and \(\pi: G \to \operatorname{End}_K(M)\) is a (global) representation such that, for all \(g \in G\), \(x \in A\), and \(m \in M\), the following condition holds:
    \[
        \pi_g(x m) = \rho_g(x) \pi_g(m).
    \]
    A morphism between two covariant representations is just a \(K\)-linear map that is simultaneously a morphism of \(A\)-modules and a morphism of representations.
    We denote by \(\textbf{CovRep}_{\rho}(G, A)\) the category of covariant representations of \(\rho\) and their morphisms.
    If there is no risk of confusion, we will use the notation \(\textbf{CovRep}(A)\) to denote the category \(\textbf{CovRep}_{\rho}(G, A)\).
    The objects of \textbf{CovRep}$(A)$ will be referred to simply as (global) \textbf{covariant representations} of \(A\).

\begin{definition}\label{d: partial covariant}
    Let $\lambda: G \to \operatorname{End}_0(A)$ be  an ideal partial representation of $G$ on \(A\).
    A \textbf{covariant representation} of $\lambda$ is a pair $(\pi, M)$, where $M$ is an $A$-module and $\pi: G \to \operatorname{End}_K(M)$ is a partial representation of $G$ such that, for all $g \in G$, $x \in A$, and $m \in M$, the following conditions hold:
    \begin{enumerate}
        \item $\pi_g(x\, m) = \lambda_g(x)\, \pi_g(m),$
        \item $\varepsilon_g(x\, m) = e_g(x)\, m = x\, \varepsilon_g(m),$
    \end{enumerate}
    where $\varepsilon_g = \pi_g \pi_{g^{-1}}$ and $e_g = \lambda_g \lambda_{g^{-1}}$.
\end{definition}

A \textbf{morphism} between two covariant representations $(\pi, M)$ and $(\xi, N)$ of $\lambda: G \to \operatorname{End}_0(A)$ is a linear map $\psi: M \to N$ that is simultaneously a morphism of  $A$-modules and a morphism of partial representations. Specifically, $\psi: M \to N$ is a morphism if, for all $g \in G$, $x \in A$, and $m \in M$, the following conditions are satisfied:
\begin{enumerate}
    \item $\psi(\pi_g(m)) = \xi_g(\psi(m)),$
    \item $\psi(x\,  m) = x\,  \psi(m).$
\end{enumerate}
The covariant representations of $\lambda: G\to \en_0(A)$, together with  their morphisms, form a category.  We denote this category by \(\textbf{CovRep}_{\lambda}(G, A)\) or simply by
\(\textbf{CovRep}_{\lambda}(A)\).

    \begin{remark}
             Suppose that $\lambda: G \to \operatorname{End}_0(A)$ is  an ideal partial representation of $G$ on a unital associative algebra $A$ and $(\pi, M)$ is a covariant representation  $\lambda $ as in De\-fi\-ni\-tion~\ref{d: partial covariant}. Let, furthermore,  $\alpha: G \curvearrowright A $ be the unital partial action  induced by \(\lambda \) and  $\phi : A \to \en _K (M)$ be the algebra homomorphism determined by the module action of $A$ on $M.$ Then it is easy to check that the pair $(\phi , \pi)  $ 
is a covariant representation of $\alpha $ in $\en _K (M),$ as defined in \cite[Definition 9.10]{Book_Exel_2017_PDSFBAA}.
    \end{remark}

    \begin{remark}
    We call a left $A$-module $M$  unital  if $M= AM  = {\rm span} \{xm \, : \, x\in A, m \in M \}.$ 
        Observe that in Definition~\ref{d: partial covariant} if we suppose that \(\lambda\) is a global representation and that \(M\) is a  unital \(A\)-module, then \((2)\) of Definition~\ref{d: partial covariant} implies that \(\pi\) is a global representation of \(G\) on \(M\).
    \end{remark}

\underline{We fix} now  an ideal partial representation  \(\lambda: G \to \operatorname{End}_0(A)\) of $G$ on $A$.
Our objective in this section is to show that \(\Lambda\) sends covariant representations of \(A\) to covariant representations of \(\Lambda(A)\). Recall that if \(A\) is a Lie (or associative) algebra, then by Theorem~\ref{t: globalization on varieties} we have that \(\Lambda(A)\) is also a Lie (or associative) algebra.

Let $(\pi, M)$ be a covariant representation of $\lambda$.  Then using (1) and (2) of Definition~\ref{d: partial covariant}, an easy calculation similar to that given for \eqref{eq: relations of pi}  shows that the equality 
\begin{equation}\label{eq: relation of pi lambda}
    \lambda_g(x) \pi_h(m) = \pi_g(x \pi_{g^{-1}h}(m)),
\end{equation}  holds for all $g,h \in G, x \in A$ and $m \in M.$ 

Observe next  that $\pi _g (M)$ is a left $A$-submodule of $M$ for any $g\in G .$ Indeed,
it follows by (2) of Definition~\ref{d: partial covariant} that for any $x\in A, m \in M$ and $g\in G$ we have that
$$x \pi _g (m) =  x \varepsilon _g (\pi _g (m))=
\varepsilon _g (x \pi _g (m))= \pi_{g} (\pi _{g^{-1}} (x \pi _g (m)))\in \pi _g(M),$$ 
as claimed. Then we also have for all $g\in G$ that 
\begin{equation}\label{eq: action of lambda_g (A) }
\lambda _g(A) M \subseteq \pi _g (M),
\end{equation}   To see this take  $x\in A, m \in M, g \in G.$ Then using  again (2) of Definition~\ref{d: partial covariant} we obtain
$$\lambda _g (x) m = e_g( \lambda _g (x)) m  
=  \lambda _g (x) \varepsilon _g (m) =
\lambda _g (x) \pi _g (\pi _{g^{-1}} (m))\in A \pi _g (M) \subseteq \pi _g (M), $$ as desired.

We proceed with the following result.
\begin{proposition}\label{p: Lambda preserve covariant representations}
    Let \((\pi, M)\) be a covariant representation of \(\lambda: G \to \operatorname{End}_0(A)\).
    Then, \(\Lambda(M)\) is a \(\Lambda(A)\)-module, and the action of \(\Lambda(A)\) on \(\Lambda(M)\) is given by
    \[
        \lfloor g, x \rfloor \lfloor h, m \rfloor = \lfloor g, x \pi_{g^{-1}h}(m) \rfloor,
    \]
    where \(\lfloor g, x \rfloor \in \Lambda(A)\) and \(\lfloor h, m \rfloor \in \Lambda(M)\).
    Moreover, the following analogue of  \eqref{eq: multiplication equivalence} holds 
    \begin{equation}\label{eq:Analogue}
    \lfloor g, x \rfloor \lfloor h, m \rfloor =  \lfloor h, \lambda_{h^{-1}g}(x) m \rfloor.
\end{equation}  for all $g,h\in G,$ $x\in A$ and $m\in M.$ 
\end{proposition}
\begin{proof} Similarly as in the proof of  Lemma~\ref{l: product} we show that the map
    \begin{align} \label{eq: Lambda(L) module}
        \begin{split}
            \Lambda(A) \times \Lambda(M) &\to \Lambda(M), \\
            (\lfloor g, x \rfloor, \lfloor h, m \rfloor) &\mapsto \lfloor g, x \rfloor \lfloor h, m \rfloor = \lfloor g, x \pi_{g^{-1}h}(m) \rfloor.
        \end{split}
    \end{align}
    is well-defined. Indeed, let $\alpha: G \curvearrowright A $  be the partial action induced by $\lambda $ and 
    $\beta: G \curvearrowright M $ that induced by $\pi .$ Denote by  \( \phi: (KG \otimes A)  \times (KG \otimes M) \to \Lambda(M) \) the map defined   by \( \phi(g \otimes x, k \otimes m) := \lfloor g, x \pi_{g^{-1}k}(m) \rfloor \). Then for any  $g,h, k \in G,$
    \( x \in A_{h^{-1}} \) and $m \in M$ we see that
    \begin{align*}
        \phi( g \otimes \alpha_{h}(x) - gh \otimes x, k \otimes m)
        &= \lfloor g,\alpha_h(x)\pi_{g^{-1}k}(m)\rfloor-\lfloor gh,x\pi_{h^{-1}g^{-1}k}(m)\rfloor \\ 
        &= \lfloor g,\lambda_h(x)\pi_{g^{-1}k}(m)\rfloor-\lfloor gh,x\pi_{h^{-1}g^{-1}k}(m)\rfloor \\
        (\text{by }\eqref{eq: relation of pi lambda})&= \lfloor g,\pi_h(x\pi_{h^{-1}g^{-1}k}(m))\rfloor-\lfloor gh,x\pi_{h^{-1}g^{-1}k}(m)\rfloor \\
        (\text{by }\eqref{eq: action of lambda_g (A) })&= \lfloor g,\beta_h(x\pi_{h^{-1}g^{-1}k}(m))\rfloor-\lfloor gh,x\pi_{h^{-1}g^{-1}k}(m)\rfloor \\
        &=\lfloor gh,x\pi_{h^{-1}g^{-1}k}(m)\rfloor-\lfloor gh,x\pi_{h^{-1}g^{-1}k}(m)\rfloor=0.
    \end{align*}
    Consequently, it follows  from Lemma~\ref{l: computation tool for Lambda}, there exists a unique linear map \( \phi_0 : \Lambda(A) \times  (KG \otimes A) \to \Lambda(
    M) \) such that  \(\phi_0( \lfloor g, x \rfloor , h \otimes m ) = \lfloor g, x \pi_{g^{-1}h}(m) \rfloor \). It is easy to see that 
    $$\phi_0(\lfloor k, x \rfloor, g \otimes \beta_{h}(m) - gh \otimes m) = 0,$$ for all $g,h,k\in G, x \in A$ and $m\in M_{h^{-1}},$ so that 
    using again Lemma~\ref{l: computation tool for Lambda},  we obtain that \eqref{eq: Lambda(L) module} is  well-defined.

    Next, in order to show \eqref{eq:Analogue} recall that   $\pi _g (M)$ is a left $A$-module and for $g, h \in G$, $x, y \in A$, and $m \in M$ compute
    \begin{align*}
        \lfloor g, x \pi_{g^{-1}h}(m) \rfloor 
        &= \lfloor g, \pi_{g^{-1}h} \pi_{h^{-1}g}(x \pi_{g^{-1}h}(m)) \rfloor  \\
        (\text{by }\eqref{eq: relation of pi lambda})&=\lfloor g, \pi_{g^{-1}h}(\lambda_{h^{-1}g}(x) m) \rfloor \\
        (\text{by }\eqref{eq: action of lambda_g (A) })&=\lfloor g, \beta_{g^{-1}h}(\lambda_{h^{-1}g}(x) m) \rfloor \\
        &= \lfloor h, \lambda_{h^{-1}g}(x) m \rfloor,
    \end{align*}
    proving  \eqref{eq:Analogue}.

    It remains to prove that  \eqref{eq: Lambda(L) module} gives rise to a left $\Lambda(A)$-module structure   on $\Lambda(M).$ Suppose  that $A$  is a Lie algebra. Then, taking $g,h,k \in G,$ $x,y \in A$ and $m \in M,$ we see, using \eqref{eq: Lambda(L) module} and \eqref{eq:Analogue}, that  
    \begin{align*}
        \lfloor g, x  \rfloor & (\lfloor h, y  \rfloor  \lfloor k, m  \rfloor)
        - \lfloor h, y  \rfloor (\lfloor g, x  \rfloor \lfloor k, m  \rfloor) \\
        &= \lfloor g, x  \rfloor \lfloor h,y  \pi_{h^{-1}k} (m) \rfloor - \lfloor h, y  \rfloor \lfloor g,x  \pi_{g^{-1}k} (m) \rfloor\\ 
        &= \lfloor g, x \pi_{g^{-1}h} (y  \pi_{h^{-1}k} (m)  ) \rfloor -  \lfloor g, \lambda_{g^{-1}h} (y) (x \pi_{g^{-1}k} (m)  ) \rfloor \\
        (\text{by }\eqref{eq: relation of pi lambda})&=  \lfloor g, x ( \lambda_{g^{-1}h} (y ) \pi_{g^{-1}k} (m)  ) \rfloor -  \lfloor g, \lambda_{g^{-1}h} (y) (x \pi_{g^{-1}k} (m)  ) \rfloor\\
        &= \lfloor g, x ( \lambda_{g^{-1}h} (y ) \pi_{g^{-1}k} (m)  )  -   \lambda_{g^{-1}h} (y) (x \pi_{g^{-1}k} (m)  ) \rfloor\\
        &= \lfloor g, (x  \lambda_{g^{-1}h} (y )) \pi_{g^{-1}k} (m)   \rfloor \\
        &= \lfloor g, x  \lambda_{g^{-1}h} (y )\rfloor  \lfloor k, m \rfloor =  
        (\lfloor g, x  \rfloor    \lfloor h, y  \rfloor )   \lfloor k, m  \rfloor ,
    \end{align*}
    so that $\Lambda(M)$ is a left module over the Lie algebra $\Lambda(A).$
    We omit the details for the associative case, since it is similar and even shorter.
\end{proof}
 
 As in the proof of Proposition~\ref{p: Lambda preserve covariant representations}, denote by  $\alpha: G \curvearrowright A $  the partial action induced by $\lambda $ and by
$\beta: G \curvearrowright M $ that induced by $\pi .$ The global action $\Lambda (\alpha )$ can be obviously seen as a (global) representation 
$\Lambda (\lambda ) : G \to  {\rm End} _{0}(\Lambda (A)),$ with $\Lambda (\lambda )_g (\lfloor h, x  \rfloor)
= \lfloor gh, x  \rfloor,$ and similarly the global action $\Lambda (\beta)$ can be considered as a (global) representation $\Lambda (\pi ) : G \to  {\rm End}_K (\Lambda (M)).$ Then it is immediate to check that $(\Lambda(\pi), \Lambda(M))$ is a covariant representation of $\Lambda (\lambda ).$

Let  $(\pi, M)$ and $(\xi, N)$ be  covariant representations of $\lambda: G \to \operatorname{End}_0(A)$, and suppose that $\psi: M \to N$ is a morphism between them. A direct computation shows that the mapping $\Lambda(\psi): \Lambda(M) \to \Lambda(N)$, defined by $\Lambda(\psi)(\lfloor g, m \rfloor) = \lfloor g, \psi(m) \rfloor$, is a morphism between the covariant representations $(\Lambda(\pi), \Lambda(M))$ and $(\Lambda(\xi), \Lambda(N))$ of $\Lambda(\lambda)$. Therefore, we have defined a functor 
 
\begin{equation}\label{eq functor Lambda}
  \Lambda: \textbf{CovRep}_{\lambda}(A) \to \textbf{CovRep}_{\Lambda (\lambda )}(\Lambda(A)).
\end{equation}
 Observe that by means of the canonical map $\iota : A \to \Lambda (A)$ any $\Lambda (A)$-module can be seen as an $A$-module.

\begin{definition}\label{d: dilated covrep}
    Let \(\lambda: G \to \operatorname{End}_0(A)\) be  an ideal  partial representation  of $G$ on $A$.
    A \textbf{dilated} covariant representation of   \( \Lambda (\lambda ) \)  is a covariant (global) representation \((\rho, W)\) of \(\Lambda(\lambda)\) equipped with a \(\Lambda(A)\)-module homomorphism \( T: W \to W \) satisfying the following conditions:
    \begin{enumerate}
        \item \( T^2 = T \),
        \item \( xw \in T(W) \) for all \( x \in A \) and \( w \in W \),
        \item \( \lfloor g, x \rfloor u = \lambda_g(x) u \) for all \( g \in G, x \in A, u \in T(W) \),
        \item \( \rho_g T \rho_{g^{-1}} T = T \rho_g T \rho_{g^{-1}} \) for all \( g \in G \).
    \end{enumerate}
\end{definition}

    \begin{remark}
        Observe that the definition of a dilated covariant representation of  \( \Lambda (\lambda ) \)   can be seen as a dilated representation of \(G\) (see \cite{Article_Alves-Batista-Vercruysse_2019_DOPROHA}), which is compatible with the globalization of the partial action on \(A\).
    \end{remark}

If \((\rho, W)\) is a \textbf{dilated} covariant representation of \(\Lambda(\lambda)\), note that condition (2) of Definition~\ref{d: dilated covrep} is equivalent to the following:  
\begin{equation}\label{eq xw}
    xw = T(xw) = xT(w) \quad \text{for all } x \in A \text{ and } w \in W.
\end{equation}

 Let \((\rho_1, W_1)\) and \((\rho_2, W_2)\) be dilated covariant representations of  \( \Lambda (\lambda ) \), equipped with their respective \(\Lambda(A)\)-module homomorphisms \(T_1: W_1 \to W_1\) and \(T_2: W_2 \to W_2\).
A morphism between these representations is a covariant representation morphism \(\psi: W_1 \to W_2\) that makes the following diagram  commutative:

\begin{center}
    \begin{tikzcd}
        {W_1} & {W_2} \\
        {W_1} & {W_2}
        \arrow["\psi", from=1-1, to=1-2]
        \arrow["{T_1}"', from=1-1, to=2-1]
        \arrow["{T_2}", from=1-2, to=2-2]
        \arrow["\psi"', from=2-1, to=2-2]
    \end{tikzcd}
\end{center}

We denote the category of dilated covariant representations of   \(\Lambda( \lambda )\) together with their morphisms by \(\textbf{Dil}_{\lambda}(G, A)\), or simply  \(\textbf{Dil}_{\lambda}(A)\) when there is no ambiguity.

\begin{lemma}\label{l: T for Lambda M}
    Let \(\lambda: G \to \operatorname{End}_0(A)\) be  an ideal partial representation  of \( G \) on \(A\) and \((\pi, M)\) be a covariant representation of \(\lambda\). Then the mapping \( T_\Lambda: \Lambda(M) \to \Lambda(M) \) given by \( \lfloor h, m \rfloor \mapsto \lfloor 1, \pi_h(m) \rfloor \) is a \(\Lambda(A)\)-module homomorphism, and \((\Lambda(\pi), \Lambda(M))\), equipped with \( T_\Lambda \), is an object of   \(\textbf{Dil}_{\lambda}(A)\). Moreover, if 
    $\varphi : (\pi _1, M_1) \to  (\pi _2, M_2) $ is morphism of covariant representations of \(\lambda\), then the map $\Lambda (\varphi ) : \Lambda (M_1) \to \Lambda (M_1),$ such that $\Lambda (\varphi )(\lfloor g,m \rfloor )= \lfloor g, \varphi (m) \rfloor,$
    for any $\in \Lambda (M_1),$ is a morphism 
    \((\Lambda(\pi), \Lambda(M)) \to (\Lambda(\pi), \Lambda(M))\) in the category \(\textbf{Dil}_{\lambda}(A)\).  Consequently, \(\Lambda\) defines a functor from  \(\textbf{CovRep}_{\lambda} (A)\) into \(\textbf{Dil}_{\lambda}(A)\).
\end{lemma}

\begin{proof}
    For \(\lfloor g, x \rfloor \in \Lambda(A)\) and \(\lfloor h, m \rfloor \in \Lambda(M)\), we compute:
    \begin{align*}
        T_\Lambda(\lfloor g,x\rfloor\lfloor h,m\rfloor)
        &= T_\Lambda(\lfloor g,x\pi_{g^{-1}h}(m)\rfloor)
        = \lfloor 1,\pi_g(x\pi_{g^{-1}h}(m)\rfloor \\
        &= \lfloor 1,\lambda_g(x)\pi_g\pi_{g^{-1}h}(m)\rfloor
        =  \lfloor 1,\lambda_g(x)\varepsilon_g\pi_{h}(m)\rfloor \\
        &=  \lfloor 1,e_g\lambda_g(x)\pi_{h}(m)\rfloor
        = \lfloor 1,\lambda_g(x)\pi_{h}(m)\rfloor \\
        &= \lfloor g,x\rfloor\lfloor 1,\pi_h(m)\rfloor 
        = \lfloor g,x\rfloor T_\Lambda(\lfloor h,m\rfloor),
    \end{align*}
    showing that \( T_\Lambda \) is a \(\Lambda(A)\)-module homomorphism.  We see from the above computation that
 $$\lfloor g,x\rfloor T_\Lambda(\lfloor h,m\rfloor) =   \lfloor 1,\lambda_g(x)\pi_{h}(m)\rfloor 
        = \lfloor 1,\lambda_g(x) \rfloor \lfloor 1,  \pi_{h}(m)\rfloor 
        = \lfloor 1,\lambda_g(x) \rfloor T_\Lambda(\lfloor h,m\rfloor), $$ which is conditions (3) of in Definition~\ref{d: dilated covrep}.      The conditions (1) and (2) of Definition~\ref{d: dilated covrep} are immediate.
    To check condition (4), let \(\lfloor h, m \rfloor \in \Lambda(M)\) and \(g \in G\). Since \(\pi_{g^{-1}}\pi_h(m) \in \im \pi_{g^{-1}}\), we have
    \begin{align*}
        \Lambda(\pi)_gT_\Lambda\Lambda(\pi)_{g^{-1}} T_\Lambda(\lfloor h,m\rfloor)
        &= \lfloor g, \pi_{g^{-1}}\pi_h(m)\rfloor
        = \lfloor 1, \pi_g\pi_{g^{-1}}\pi_h(m)\rfloor \\
        &= \lfloor 1,\pi_g \pi_{g^{-1}h}(m)\rfloor
        = T_\Lambda\Lambda(\pi)_{g} T_\Lambda\Lambda(\pi)_{g^{-1}}(\lfloor h,m\rfloor).
    \end{align*}
    Thus, condition (4) is satisfied.
    
  The verification of the fact that $\Lambda $ takes morphisms  from \(\textbf{CovRep}_{\lambda} (A)\) to morphsims in  \(\textbf{Dil}_{\lambda}(A)\) is straightforward. 
\end{proof}

Suppose \((\rho, W)\), equipped with the \(\Lambda(A)\)-module homomorphism \(T: W \to W\), is an object of    \(\textbf{Dil}_{\lambda}(A)\).  For each \( g \in G \), define \(\xi_g = T\rho_g|_{T(W)}\), yielding a mapping \(\xi: G \to \operatorname{End}_K(T(W))\). For any \(g \in G\) and \(u \in T(W)\), the following holds
\[
    \xi_g \xi_{g^{-1}}(u) = T\rho_g T\rho_{g^{-1}}(u) = \rho_g T\rho_{g^{-1}} T(u) = \rho_g T\rho_{g^{-1}}(u) = \rho_g \xi_{g^{-1}}(u),
\]
which implies
\begin{equation}\label{eq xi}
    \xi_g \xi_{g^{-1}} = \rho_g \xi_{g^{-1}}.
\end{equation}
Additionally, for \(x \in A\)  and \(u \in T(W)\), condition (3) of Definition~\ref{d: dilated covrep} gives
\begin{equation}\label{eq relation rho}
    \rho_g(\lambda_{g^{-1}}(x)u) = x\rho_g(u).
\end{equation}

\begin{lemma}\label{le xi W}
    Let \((\rho, W)\) be a dilated covariant representation of \(\Lambda(\lambda)\), equipped with the \(\Lambda(A)\)-module homomorphism \(T: W \to W\), and let \(\xi\) be defined as above. Then, the pair \((\xi, T(W))\) is a covariant representation of \(\lambda\).
\end{lemma}

\begin{proof}
    First, we verify that \(\xi: G \to \operatorname{End}_K(T(W))\) is a partial representation of \(G\). It is immediate that \(\xi_1 = id_{T(W)}\). Now, let \(g, h \in G\). Then
    \[
        \xi_g \xi_h \xi_{h^{-1}} \overset{\scriptsize\eqref{eq xi}}{=} T\rho_g \rho_h \xi_{h^{-1}} = T\rho_{gh} \xi_{h^{-1}}  = \xi_{gh} \xi_{h^{-1}}.
    \]
   On the other hand, for any $g,h \in G$ and   \(u \in T(W)\),
    \begin{align*}
        \xi _{g^{-1}} \xi_g \xi_h (u)
        \overset{\scriptsize\eqref{eq xi}}{=} \rho _{g^{-1}} \xi_g \xi_h (u)
        &= (\rho _{g^{-1}} T \rho_g T) \rho_h (u)\\
        &\overset{(\flat)}{=}  (T \rho _{g^{-1}} T \rho_g)  \rho_h (u)
        = T \rho _{g^{-1}} T \rho_{gh}  (u)
        =\xi_{g^{-1}} \xi_{gh}(u),
    \end{align*}
    where \((\flat)\) follows from condition (4) of Definition~\ref{d: dilated covrep}.
    Thus, \(\xi_{g^{-1}} \xi_g \xi_h = \xi_{g^{-1}} \xi_{gh}\).
    This proves that \(\xi\) is indeed a partial representation of \(G\).

    Now we show that \((\xi, T(W))\) satisfies conditions (1) and (2) in Definition~\ref{d: partial covariant}. Let \(g \in G\), \(x \in A\), and \(u \in T(W)\). Using condition (3) of Definition~\ref{d: dilated covrep}, we compute:
    \begin{multline*}
        \xi_g(xu)=\xi_g(\lfloor 1,x\rfloor u)=T\rho_g(\lfloor 1,x\rfloor u)=T(\Lambda(\lambda)_g(\lfloor 1,x\rfloor)\rho_g(u))=\\= T(\lfloor g,x\rfloor \rho_g(u))=\lfloor g,x\rfloor T(\rho_g(u))=\lambda_g(x)\xi_g(u),
    \end{multline*}
    showing condition~(1). Additionally, note that
    \[
        \xi_g \xi_{g^{-1}}(xu) \overset{\scriptsize\eqref{eq xi}}{=} \rho_g \xi_{g^{-1}}(xu) = \rho_g(\lambda_{g^{-1}}(x)\xi_{g^{-1}}(u)) \overset{\scriptsize\eqref{eq relation rho}}{=} x\rho_g \xi_{g^{-1}}(u) \overset{\scriptsize\eqref{eq xi}}{=} x\xi_g \xi_{g^{-1}}(u),
    \]
    and on the other hand
    \begin{multline*}
        \lambda_g\lambda_{g^{-1}}(x)u=\lfloor g,\lambda_{g^{-1}}(x)\rfloor u= \rho_g(\Lambda(\lambda)_{g^{-1}}(\lfloor g,\lambda_{g^{-1}}(x)\rfloor)\rho_{g^{-1}}(u))=\rho_g(\lambda_{g^{-1}}(x)\rho_{g^{-1}}(u))=\\\overset{\scriptsize\eqref{eq xw}}{=}\rho_g T(\lambda_{g^{-1}}(x)\rho_{g^{-1}}(u))=\rho_g(\lambda_{g^{-1}}(x)T(\rho_{g^{-1}}(u))\overset{\scriptsize\eqref{eq relation rho}}{=}x\rho_g\xi_{g^{-1}}(u)\overset{\scriptsize\eqref{eq xi}}{=}x\xi_g\xi_{g^{-1}}(u),
    \end{multline*}
    and condition~(2) is proven.
\end{proof}

Let \((\rho_1, W_1)\) and \((\rho_2, W_2)\) be dilated covariant representations of \(\Lambda(\lambda)\),   equipped with their respective \(\Lambda(A)\)-module homomorphisms \(T_1: W_1 \to W_1\) and \(T_2: W_2 \to W_2\). Suppose  that  \(\psi: W_1 \to W_2\) is a morphism between these representations. A straightforward computation shows that the restriction \(\psi|_{T_1(W_1)}: T_1(W_1) \to T_2(W_2)\) is a morphism between the covariant representations \((\xi^1, T_1(W_1))\) and \((\xi^2, T_2(W_2))\) of \(\lambda\), as described in Lemma~\ref{le xi W}. This observation allows  us to define a functor:
\[
    \F: \textbf{Dil}_{\lambda}(A) \to \textbf{CovRep}_{\lambda}(A),
\]
sending the dilated covariant representation \((\rho, W)\) to the covariant representation \((\xi, T(W))\) (as given in Lemma~\ref{le xi W}) and a morphism \(\psi: W_1 \to W_2\) to its restriction \(\psi|_{T_1(W_1)}\).

\begin{theorem}
    Let \(\lambda: G \to \operatorname{End}_0(A)\) be  an ideal  partial representation  of \(G\) on \(A\). Then the functor 
    \(
    \Lambda: \textbf{CovRep}_{\lambda}(A) \to \textbf{Dil}_{\lambda}(A) 
    \)
    is left adjoint to the functor 
    \(
    \F: \textbf{Dil}_{\lambda}(A) \to \textbf{CovRep}_{\lambda}(A).
    \)
\end{theorem}

\begin{proof}
    Let \((\pi, M)\) be an object in \(\textbf{CovRep}_{\lambda}(A)\), and let \((\rho, W)\), equipped with the \(\Lambda(\lambda)\)-module homomorphism \(T: W \to W\), be an object in \(\textbf{Dil}_{\lambda}(A)\). Recall, from Lemma~\ref{l: T for Lambda M}, that \((\Lambda(\pi), \Lambda(M))\) is equipped with the \(\Lambda(A)\)-module homomorphism \(T_\Lambda: \Lambda(M) \to \Lambda(M)\) defined by \(\lfloor g, m \rfloor \mapsto \lfloor 1, \pi_g(m) \rfloor\).

    Define the mapping
    \[
        \eta: \Hom((\Lambda(\pi), \Lambda(M)), (\rho, W)) \to \Hom((\pi, M), (\xi, T(W))),
    \]
    which sends a morphism \(\gamma: \Lambda(M) \to W\) to the morphism \(\eta(\gamma): M \to T(W)\) given by \(m \mapsto \gamma(\lfloor 1, m \rfloor)\). We first verify that \(\eta\) is well-defined, i.e., that \(\eta(\gamma): M \to T(W)\) is a morphism of  dilated covariant representations of \(\lambda\). 
   Note that
    $$T \eta(\gamma)(m) = T \gamma (\lfloor 1, m \rfloor) = \gamma T_{\Lambda}  (\lfloor 1, m \rfloor) = \gamma   (\lfloor 1, m \rfloor) =  \eta(\gamma)(m) $$ for any $m\in M$  so that
    $ \eta(\gamma)(M) \subseteq T(W).$ 
    Next, for  \(x \in A\), \(m \in M\), and \(g \in G\), we compute
    \[
        \eta(\gamma)(xm) = \gamma(\lfloor 1, xm \rfloor) = \gamma(\lfloor 1, x \rfloor \lfloor 1, m \rfloor) = \lfloor 1, x \rfloor \gamma(\lfloor 1, m \rfloor) = x\gamma(\lfloor 1, m \rfloor)
        = x (\eta(\gamma)(m)),
    \]
    showing that \(\eta(\gamma)\) is a morphism of \(A\)-modules. Now, for \(g \in G\),
\begin{multline*}
\eta(\gamma)(\pi_g(m)) = \gamma(\lfloor 1, \pi_g(m) \rfloor) = \gamma(T_\Lambda(\lfloor g, m \rfloor)) =\\= T\gamma(\Lambda(\lambda)_g(\lfloor 1, m \rfloor))
= T\rho_g\gamma(\lfloor 1, m \rfloor) = \xi_g(\eta(\gamma)(m)),
\end{multline*}    
    proving that \(\eta(\gamma)\) is a morphism of partial representations  of $G.$

    Next, define the mapping
    \[
        \delta: \Hom((\pi, M), (\xi, T(W))) \to \Hom((\Lambda(\pi), \Lambda(M)), (\rho, W)),
    \]
    sending a morphism \(\psi: M \to T(W)\) to the morphism \(\delta(\psi): \Lambda(M) \to W\) defined by \(\lfloor h, m \rfloor \mapsto \rho_h\psi(m)\). We verify that \(\delta\) is well-defined, i.e., that \(\delta(\psi)\) is a morphism of covariant representations of \(\Lambda(\lambda)\). For \(\lfloor g, x \rfloor \in \Lambda(A)\) and \(\lfloor h, m \rfloor \in \Lambda(M)\), we have
    \begin{multline*}
        \delta(\psi)(\lfloor g,x\rfloor \lfloor h,m\rfloor)=\delta(\psi)(\lfloor g,x\pi_{g^{-1}h}(m)\rfloor)=\rho_g \psi(x\pi_{g^{-1}h}(m))=\rho_g(x\xi_{g^{-1}h}\psi(m))= \\ =\rho_g(xT\rho_{g^{-1}h}\psi(m)) \overset{\scriptsize\eqref{eq xw}}{=} \rho_g(\lfloor 1,x\rfloor\rho_{g^{-1}h}\psi(m))=\Lambda(\lambda)_g(\lfloor 1,x\rfloor) \rho_h\psi(m)=\lfloor g,x\rfloor \delta(\psi)(\lfloor h,m\rfloor),
    \end{multline*}
    proving that \(\delta(\psi)\) is a morphism of \(\Lambda(A)\)-modules. Additionally,
    \[
        \delta(\psi)(\Lambda(\pi)_g(\lfloor h, m \rfloor)) = \delta(\psi)(\lfloor gh, m \rfloor) = \rho_{gh}\psi(m) = \rho_g(\rho_h\psi(m)) = \rho_g\delta(\psi)(\lfloor h, m \rfloor),
    \]
    verifying that \(\delta(\psi)\) is a morphism of covariant representations. Finally, we check that \(\delta(\psi) \circ T_\Lambda = T \circ \delta(\psi)\):
\begin{multline*}
\delta(\psi)(T_\Lambda(\lfloor h, m \rfloor)) = \delta(\psi)(\lfloor 1, \pi_h(m) \rfloor) = \psi(\pi_h(m)) =\\= \xi_h(\psi(m)) = T\rho_h(\psi(m)) = T\delta(\psi)(\lfloor h, m \rfloor).
\end{multline*}    
A straightforward computation shows that \(\delta\) is the inverse of \(\eta\), and thus \(\eta\) is a bijection.

    To conclude the proof of the theorem, we must show that the bijection \(\eta\) is natural. This requires verifying that for any morphism \(\gamma: M \to N\) in  \(\textbf{CovRep}_{\lambda}(A)\)  and any morphism \(\psi: W \to V\) in  \(\textbf{Dil}_{\lambda}(A)\), the  diagrams 
    \[
        \begin{tikzcd}
            \Hom((\Lambda({\pi}'), \Lambda(N)), (\rho, W)) 
            \arrow[r, "\Lambda(\gamma)^*"] 
            \arrow[d, "\eta"'] 
            & \Hom((\Lambda(\pi), \Lambda(M)), (\rho, W))  
            \arrow[d, "\eta"']  \\
            \Hom(({\pi}', N), (\xi, T(W))) 
            \arrow[r, "\gamma^*"] 
            & \Hom((\pi, M), (\xi, T(W)))  
        \end{tikzcd}
    \]

    and 
    
   \[
        \begin{tikzcd}
           \Hom((\Lambda (\pi), \Lambda(M)), (\rho, W)) 
            \arrow[r, "\psi_*"] 
            \arrow[d, "\eta"'] 
            & \Hom((\Lambda (\pi), \Lambda(M)), (\rho', V)) 
            \arrow[d, "\eta"] \\
            \Hom((\pi , M), (\xi, T(W))) 
            \arrow[r, "\F(\psi)_*"] 
            & \Hom((\pi , M), (\xi', T(V)))
        \end{tikzcd}
    \]  
     are commutative, where 
     \begin{align*}
     &\Lambda(\gamma)^* = \Hom (\Lambda (\gamma ), (\rho , W)), 
     {\gamma} ^*= \Hom (\gamma , (\xi , T(W))),\\
     &\psi _* = \Hom ((\Lambda(\pi), \Lambda (M)), \psi), \F(\psi)_* = \Hom ((\pi,M), \F (\psi) ) .
    \end{align*} 
    

   For let \(\phi \in \Hom((\Lambda({\pi}'), \Lambda(N)), (\rho, W))\). For any \(m \in M\), we have
    \[
        \eta(\Lambda(\gamma)^*(\phi))(m) = \eta(\phi \circ \Lambda(\gamma))(m) = (\phi \circ \Lambda(\gamma))(\lfloor 1, m \rfloor) = \phi(\lfloor 1, \gamma(m) \rfloor),
    \]
    and
    \[
        \gamma^*(\eta(\phi))(m) = (\eta(\phi) \circ \gamma)(m) = \eta(\phi)(\gamma(m)) = \phi(\lfloor 1, \gamma(m) \rfloor),
    \]
     showing  the commutativity of the  first diagram. Now, let \(\varphi \in \Hom((\Lambda(\pi), \Lambda(M)), (\rho, W))\). For any \(m \in M\), we see that
    \[
        \eta(\psi_*(\varphi))(m) = \eta(\psi \circ \varphi)(m) = (\psi \circ \varphi)(\lfloor 1, m \rfloor) = \psi(\varphi(\lfloor 1, m \rfloor)),
    \]
    and
    \[
        \F(\psi)_*(\eta(\varphi))(m) = (\F(\psi) \circ \eta(\varphi))(m) = (\psi|_{T(W)} \circ \eta(\varphi))(m) = \psi(\varphi(\lfloor 1, m \rfloor)),
    \]
    proving that the second diagram is also commutative.
\end{proof}

\section{Covariant representations and semidirect products of Lie algebras}

Let $L$ and $M$ be Lie algebras. We say that $L$ acts by derivations on $M$  if there exists a $K$-bilinear map 
\(
L \times M \to M, (x, m) \mapsto x  m,
\)
such that, for every $x,y\in L$ and $m,n\in M$, the following conditions are satisfied: 
\begin{align} \label{eq:action by derivations}
    \begin{split}
        (x \cdot y) m &= x (y  m) - y  (x  m),\\
        x(m\cdot n) &=(xm)\cdot n+m(x\cdot n).
    \end{split}
\end{align}
Equivalently, $L$ acts  by derivations on $M$ if there exists a Lie algebra homomorphism $$L\to \der(M),$$ where $\der(M)$ is the Lie algebra of derivations of $M$.  Note that the first equality in \eqref{eq:action by derivations} says that $M$ is a left $L$-module.

We recall that if $L$ acts by derivations on $M$, then the semidirect product of $L$ and $M$, which is denoted by $L\ltimes M$, is  the Lie algebra whose underlying vector space is $L \times M$, and  whose  product is given by
\begin{equation}\label{eq semidirec bracket}
    (x,m)\cdot(y,n)=(x\cdot y,m\cdot n+xn-ym).
\end{equation}

\begin{proposition}\label{p:pr of semidirect product}
    Let $L$ and $M$ be Lie algebras, with $L$ acting by derivations  on $M$. Suppose that \(\lambda: G \to \operatorname{End}_0(L)\) and $\pi: G\to \operatorname{End}_0(M)$ are  ideal partial representations, such that  \((\pi, M)\) is a covariant representation of \(\lambda\). Then, the following  hold:
    \begin{enumerate}
        \item The mapping $\lambda\times \pi: G\to \operatorname{End}_0 (L \ltimes M)$, defined by $(\lambda\times \pi)_g: (x,m)\mapsto (\lambda_g(x),\pi_g(m))$, is  an ideal partial representation on \(L\ltimes M\).
        \item The map  
            $\Lambda(L) \times \Lambda(M) \to \Lambda(M),$ defined  by $ 
            (\lfloor g, x \rfloor, \lfloor h, m \rfloor) \mapsto \lfloor g, x \rfloor \lfloor h, m \rfloor = \lfloor g, x \pi_{g^{-1}h}(m) \rfloor$ (as in~\eqref{eq: Lambda(L) module}), is an action by derivations of $\Lambda(L)$ on $\Lambda(M)$. Consequently, the mapping $$\Lambda(\lambda)\times \Lambda(\pi): G\to  \operatorname{End}_{0}(\Lambda(L)\ltimes \Lambda(M))$$ is a (global) representation of $G$ on $\Lambda(L)\ltimes \Lambda(M)$.
    \end{enumerate}
\end{proposition}
\begin{proof}\,

    \begin{enumerate}[(1)]
        \item The only non-trivial condition required to prove  for this item is  that \(\im (\lambda \times \pi)_g=\im\lambda_g\times \im\pi_g\)  is an ideal of $L\ltimes M$. Recall that an element $(x,m)$ belongs to $\im(\lambda\times \pi)_g$ if and only if $$(\lambda\times \pi)_g(\lambda\times \pi)_{g^{-1}}(x,m)=(x,m).$$ Denote $\varepsilon_g=\pi_g\pi_{g^{-1}}$ and $e_g=\lambda_g\lambda_{g^{-1}}$. Let $(x,m)\in \im(\lambda\times \pi)_g$ and $(y,n)\in L\ltimes M$.  In view of \eqref{eq semidirec bracket}, in order to show that $(x,m)\cdot(y,n) \in \im(\lambda\times \pi)_g,$
         it suffices to verify that $\varepsilon_g(xn-ym)=xn-ym$. Since $(\pi,M)$ is a covariant representation of $\lambda$, we obtain 
            $$\varepsilon_g(xm-yn)=\varepsilon_g(xn)-\varepsilon_g(ym)=
            e_g (x) n-y\varepsilon_g (m)= xn-ym.$$
        \item Since  by Proposition~\ref{p: Lambda preserve covariant representations}  $\Lambda(M)$ is an $\Lambda(L)$-module with this action,  we only need to verify the second equality of~\eqref{eq:action by derivations}. Let $x,\in L$, $g, h, k \in G$, and $m,n \in M$. Then
            \begin{align*}
                \lfloor g, x \rfloor (\lfloor h, m \rfloor\cdot \lfloor k, n \rfloor)&= \lfloor g, x \rfloor \lfloor h, m \cdot \pi_{h^{-1}k}(n) \rfloor\\ & = \lfloor g, x \pi_{g^{-1}h}(m \cdot \pi_{h^{-1}k}(n))\rfloor \\ &=\lfloor g, x (\pi_{g^{-1}h}(m) \cdot  \pi_{g^{-1}k}(n))\rfloor \\&=  \lfloor g, (x \pi_{g^{-1}h}(m)) \cdot  \pi_{g^{-1}k}(n)\rfloor + \lfloor g, \pi_{g^{-1}h}(m) \cdot  (x\pi_{g^{-1}k}(n))\rfloor \\&= (\lfloor g, x \rfloor \lfloor h, m \rfloor)\cdot \lfloor k, n \rfloor + \lfloor h, m \rfloor\cdot (\lfloor g, x \rfloor \lfloor k, n \rfloor),
            \end{align*}
            as required. The last assertion follows from item (1), as $(\Lambda(\pi), \Lambda(M))$ is a (global) covariant representation of $\Lambda(\lambda)$.
    \end{enumerate}
\end{proof}

\begin{theorem}
    Let $L$ and $M$ be Lie algebras, with $L$ acting by derivations on $M$. Suppose that \(\lambda: G \to \operatorname{End}_0(L)\) and $\pi: G\to \operatorname{End}_0(M)$ are  ideal  partial representations  such that  \((\pi, M)\) is a covariant representation of \(\lambda\). Then, the actions $\Lambda(\lambda\times \pi)$ and $\Lambda(\lambda)\times \Lambda(\pi)$ are isomorphic.
    In particular: 
    \[
       \Lambda(L\ltimes M)\cong \Lambda(L)\ltimes\Lambda(M).
    \]
\end{theorem}

\begin{proof}
    We will verify that the mapping 
    \begin{align*}
        \psi: \Lambda(L\ltimes M) &\to \Lambda(L)\ltimes\Lambda(M)\\
        \lfloor g,(x,m)\rfloor &\mapsto (\lfloor g,x\rfloor, \lfloor g,m\rfloor)
    \end{align*}
    is an isomorphism between the actions $\Lambda(\lambda\times \pi)$ and $\Lambda(\lambda)\times \Lambda(\pi)$.  It is easily seen that $\psi $ is a well defined linear mapping. Let us see that $\psi$ is a homomorphims of Lie algebras. Let $\lfloor g,(x,m)\rfloor, \lfloor h,(y,n)\rfloor \in \Lambda(L\ltimes M)$. First we calculate 
    \begin{align*}
        \lfloor g,(x,m)\rfloor\cdot\lfloor h,(y,n)\rfloor &=\lfloor g,(x,m)\cdot(\lambda_{g^{-1}h}(y),\pi_{g^{-1}h}(n))\rfloor\\ &=\lfloor g,(x\cdot \lambda_{g^{-1}h}(y), m\cdot\pi_{g^{-1}h}(n)+x\pi_{g^{-1}h}(n)- \lambda_{g^{-1}h}(y)m)\rfloor.
    \end{align*} 
    With this in mind we get
\begin{multline*}
 \psi(\lfloor g,(x,m)\rfloor\cdot\lfloor h,(y,n)\rfloor )=(\lfloor g,x\cdot \lambda_{g^{-1}h}(y)\rfloor, \lfloor g, m\cdot\pi_{g^{-1}h}(n)\rfloor +\\ +\lfloor g, x\pi_{g^{-1}h}(n)\rfloor- \lfloor g,\lambda_{g^{-1}h}(y)m)\rfloor).
\end{multline*}    
    On the other hand
\begin{multline*}
\hspace{-0.3cm}
 \psi(\lfloor g,(x,m)\rfloor) \cdot \psi (\lfloor h,(y,n)\rfloor )=\, (\lfloor g,x\rfloor, \lfloor g,m\rfloor)\cdot (\lfloor h,y\rfloor, \lfloor h,n\rfloor)\,=\\=  (\lfloor g,x\rfloor\cdot \lfloor h,y\rfloor, \lfloor g,m\rfloor\cdot \lfloor h,n\rfloor +\lfloor g,x\rfloor\lfloor h,n\rfloor - \lfloor h,y\rfloor \lfloor g,m\rfloor)=\\=\lfloor g,x\cdot \lambda_{g^{-1}h}(y)\rfloor +\lfloor g, m\cdot\pi_{g^{-1}h}(n)\rfloor+\lfloor g, x\pi_{g^{-1}h}(n)\rfloor-\lfloor g,\lambda_{g^{-1}h}(y)m)\rfloor,
\end{multline*}    
showing that $\psi$ is a Lie algebra homomorphism. 

    It is straightforward to check that $\psi$ is $G$-equivariant. To verify that $\psi$ is bijetive, consider the mapping $\varphi: \Lambda(L)\ltimes \Lambda(M)\to \Lambda(L\ltimes M)$ defined by $$(\lfloor g,x\rfloor, \lfloor h,m\rfloor)\mapsto \lfloor g,(x,0)\rfloor+ \lfloor h,(0,m)\rfloor.$$  It is directly seen that $\varphi $ is well defined. The equalities 
\begin{align*}
\psi\varphi(\lfloor g,x\rfloor, \lfloor h,m\rfloor)&=\psi(\lfloor g,(x,0)\rfloor+ \lfloor h,(0,m)\rfloor)\\ &=(\lfloor g,x\rfloor,\lfloor g,0\rfloor)+(\lfloor h,0\rfloor,\lfloor h,m\rfloor)=(\lfloor g,x\rfloor,\lfloor h,m\rfloor)
\end{align*}    
    and 
\begin{align*}
\varphi\psi(\lfloor g,(x,m)\rfloor) &=\varphi(\lfloor g,x\rfloor, \lfloor g,m\rfloor)=\lfloor g,(x,0)\rfloor+ \lfloor g,(0,m)\rfloor\\&=\lfloor g,(x,0)+(0,m)\rfloor=\lfloor g,(x,m)\rfloor
\end{align*}
imply that $\varphi$ is the inverse of $\psi$, and therefore $\psi$ is bijective.
\end{proof}

\section*{Acknowledgments} 
 The first named author was partially supported by Funda\c c\~ao de Amparo \`a Pesquisa do Estado de S\~ao Paulo (Fapesp), process n°:  2020/16594-0, and by  Conselho Nacional de Desenvolvimento Cient\'{\i}fico e Tecnol{\'o}gico (CNPq), process n°: 312683/2021-9. The third named author was supported by PRPI da Universidade de São Paulo, process n°: 22.1.09345.01.2.

\bibliographystyle{abbrv}
\bibliography{azu.bib}

\end{document}